\documentclass{amsart}
\usepackage[utf8]{inputenc}
\usepackage{graphicx,graphics}

\usepackage{amsmath,amssymb}

\usepackage{epstopdf}

\usepackage{enumerate}
\usepackage{mathtools}
\usepackage{color,xcolor} 
\usepackage{mathrsfs,bbm} 

\usepackage{hyperref}

\usepackage{mathpazo}
\usepackage{dsfont}

\usepackage{epstopdf}

\usepackage{algorithm}
\usepackage{algorithmic}
\usepackage{enumerate}
\usepackage{mathtools}
\usepackage{xcolor} 

\usepackage{mathrsfs} 

\newcommand{\R}{\mathbb{R}}
\newcommand{\C}{\mathbb{C}}
\newcommand{\bC}{\boldsymbol{C}}

\DeclareMathAlphabet\mathbfcal{OMS}{cmsy}{b}{n}

\newcommand{\mXG}{{\mathbf{X}_\Gamma}}
\newcommand{\mVG}{{\mathbf{V}_\Gamma}}

\renewcommand{\Re}{\mathrm{Re}\,}

\newcommand{\pt}{\partial_t}

\newcommand{\tauh}{^{\tau,h}}
\newcommand{\tauhn}{^{\tau,h}_n}

\newcommand{\Einc}{\E^\textnormal{inc}}
\newcommand{\Hinc}{\H^\textnormal{inc}}

\newcommand{\bnu}{\boldsymb \nu}

\newcommand{\bpsi}{{\boldsymbol{\psi}}}

\newcommand{\bxi}{\boldsymbol{\xi}}
\newcommand{\bseta}{\boldsymbol{\eta}}

\newcommand{\bwvarphi}{\widehat{\boldsymbol{\varphi}}}
\newcommand{\bwpsi}{\widehat{\boldsymbol{\psi}}}

\newcommand{\bd}{\textbf{d}}

\newcommand{\bphi}{\boldsymbol{\phi}}
\newcommand{\bvarphi}{{\boldsymbol{\varphi}}}

\newcommand{\bV}{\boldsymb{V}}
\newcommand{\bK}{\boldsymb{K}}

\newcommand{\bI}{\boldsymb{I}}

\newcommand{\divG}{\operatorname{div}_\Gamma}
\DeclareMathOperator{\curl}{\mathbf{curl}}

\newcommand{\norm}[1]{\left\lVert#1\right\rVert}

\newcommand{\abs}[1]{\left|#1\right|}
\newcommand{\jmp}[1]{[#1]}
\newcommand{\avg}[1]{\{#1\}}

\newcommand{\Cald}{\boldsymb C}
\newcommand{\Caldimp}{{\boldsymb C}_{\text{imp}}}
\renewcommand{\Re}{\operatorname{Re}}
\newcommand{\dist}{\operatorname{dist}}
\newcommand{\inc}{^\textnormal{inc}}

\newcommand{\tot}{^\textnormal{tot}}

\newcommand{\boldsymb}{\boldsymbol}
\newcommand{\E}{\boldsymb E}
\newcommand{\V}{\boldsymb V}
\newcommand{\X}{\boldsymb X}
\newcommand{\bL}{\boldsymb L}

\renewcommand{\H}{\boldsymb H}
\newcommand{\wE}{\widehat{\boldsymb E}}
\newcommand{\wH}{\widehat{\boldsymb H}}
\renewcommand{\a}{\boldsymb a}
\renewcommand{\b}{\boldsymb b}
\renewcommand{\c}{\boldsymb c}

\newcommand{\p}{\boldsymb p}

\newcommand{\ba}{\boldsymb a}
\newcommand{\be}{\boldsymb e}
\newcommand{\bu}{\boldsymb u}
\newcommand{\bv}{\boldsymb v}

\newcommand{\bx}{\boldsymb x}
\newcommand{\by}{\boldsymb y}

\newcommand{\vtn}{\underline{t_n}}

\newcommand{\bigcupdot}{\bigcup\hspace{-6pt}\cdotp\hspace{4pt}}

\newcommand{\eps}{\varepsilon}
\newcommand{\bone}{\mathbbm{1}}
\newcommand{\Ctrace}{C_\Gamma}

\newcommand{\ecl}{\color{black}}

\newcommand{\brev}{\color{black}}
\newcommand{\erev}{\color{black}}

\newtheorem{proposition}{Proposition}
\newtheorem{lemma}{Lemma}

\newtheorem{theorem}{Theorem}

\theoremstyle{custom1}
\newtheorem{definition}{Definition}[section]
\newtheorem{remark}[definition]{Remark}
\theoremstyle{custom2}

\author{J\"org Nick}
\address{
	Seminar for Applied Mathematics,
	ETH Zürich, 8092 Zürich, Switzerland 
}
\email{joerg.nick@math.ethz.ch}

\begin{document}

\title[Electromagnetic scattering with nonlinear boundary conditions]{Numerical analysis for electromagnetic scattering with nonlinear boundary conditions
}

%
%
%
%
%
\maketitle
\begin{abstract}
This work studies time-dependent electromagnetic scattering from obstacles whose interaction with the wave is fully determined by a nonlinear boundary condition.  In particular, the boundary condition studied in this work enforces a power law type relation between the electric and magnetic field along the boundary.
Based on time-dependent \brev jump relations \erev of classical boundary operators, we derive a nonlinear system of time-dependent boundary integral equations that determines the tangential traces of the scattered electric and magnetic fields. 
These fields can subsequently be computed at arbitrary points in the exterior domain by evaluating a time-dependent representation formula. 

Fully discrete schemes are obtained by discretising the nonlinear system of  boundary integral equations with Runge--Kutta based convolution quadrature in time and Raviart--Thomas boundary elements in space. Error bounds with explicitly stated convergence rates are proven, under the assumption of sufficient regularity of the exact solution. The error analysis is conducted through novel techniques based on time-discrete transmission problems and the use of a new discrete partial integration inequality. Numerical experiments illustrate the use of the proposed method and provide empirical convergence rates. 
\end{abstract}
\section{Introduction}

This work proposes and studies numerical schemes, which discretize Maxwell's equations in the context of wave scattering, where the interaction of the wave with the obstacle is governed by a nonlinear boundary condition. 

Asymptotic analysis of small scale effects on the boundary of obstacles, typically arising from thin coatings around the scatterer, yield a large variety of boundary conditions of practical interest. Starting from \cite{EN93}, such asymptotic models have been studied extensively  in the time-harmonic setting, for example in \cite{HJN05}, \cite{HJN08} or \cite{DHJ06}. 

When the material properties of the coating \brev exhibit \erev nonlinear phenomena, the derived boundary conditions may be nonlinear as well, as demonstrated in \cite{HJ01} and \cite{HJ02}, which derive asymptotic models for thin ferromagnetic coatings.
The presence of nonlinear phenomena naturally prohibits the use of time-harmonic techniques, which significantly complicates both the analysis and the numerical treatment of such problems. 
Consequently, the existing literature on nonlinear scattering problems is scarce and the numerical treatment of nonlinear scattering is rarely considered. 
The acoustic wave equation with nonlinear boundary conditions in the context of scattering has been analyzed in \cite{BanjaiRieder} and \cite{BanjaiLubich2019}.

\subsection{Problem setting}
Let  $\Omega$ denote an exterior Lipschitz domain, which is assumed to be
the complement of one or several bounded domains.  
The total electric field $\E^\text{tot}(x,t)$ and the total magnetic field $\H^\text{tot}(x,t)$ are said to be solutions of \emph{ Maxwell's equations} if
\begin{align}\label{maxwell-tot}
\begin{split}
\varepsilon\,\pt \E^\text{tot}-\curl \H^\text{tot} &= 0 \quad  \\ 
\mu \,\pt \H^\text{tot} + \curl \E^\text{tot} &=0
\end{split} \quad \text{ in the exterior domain }\Omega.
\end{align}
The permittivity $\varepsilon$ and the permeability $\mu$ in~$\Omega$ are known positive constants, which describe the material properties of the free space around the scatterer.

The total fields are initially, at time $t = 0$, assumed to have support away from the boundary $\Gamma$. The initial values are further assumed to be \brev derived from \erev of incident fields $(\E^\text{inc},\H^\text{inc})$, solutions to Maxwell's equations on the full space $\mathbb R^3$. This setting then allows for a formal decomposition of the total fields into unknown scattered fields, which initially vanish, and the known \brev incident \erev fields.

Throughout the paper, the wave speed $c$ is assumed to be set to one, which is always achieved by rescaling the time variable $t \to ct$. Consequently, the product of the permittivity $\varepsilon$ and permeability $\mu$ is assumed to be normalized, since 
$$
\eps\mu=c^{-2}=1.
$$
To completely avoid the occurence of physical constants in the analysis, we further employ the rescaling $\mu\H\to  \H$. This rescaled field $ \mu \H$ is, in the context of physics literature, also referred to as the magnetic field $\boldsymb B$. 
Applying this assumption and rescaling yields time-dependent Maxwell's equations without the physical constants $\varepsilon$ and $\mu$, which read
\begin{align}\label{maxwell-tot-EB}
\begin{split}
\pt \E\tot-\curl \H\tot &= 0 \quad  \\
\pt \H\tot + \curl \E\tot &=0
\end{split} \quad \text{ in the exterior domain }\Omega.
\end{align}


\noindent 
 The \emph{nonlinear boundary condition} studied here enforces a nonlinear relation between the traces of the electromagnetic fields and reads
\begin{align}\label{GIBC_tot}
 \E^\text{tot}\times \bnu +   \a(\H^\text{tot}\times \bnu)\times \bnu  = 0\quad \text{ on }\Gamma=\partial\Omega,
\end{align}
where $\bnu$ denotes the outer unit normal vector. Note that the rescaling of $\H$ with regards to the physical constant $\mu$ is, in this formulation of the boundary condition, assumed to be incorporated into the \brev pointwise \erev nonlinearity $\ba$.

Despite the \brev apparent simplicity of this problem formulation\erev, serious challenges arise both in the numerical treatment and analysis of the described problem, due to the nonlinearity of the boundary condition.   
Throughout the paper, the nonlinearity is restricted to be a power law of the following type
\begin{equation}\label{power-law}
\a (\bx )=\abs{\bx }^{\alpha-1}\bx \quad \text{for all } \bx \in\mathbb R^3,
\end{equation}
for some fixed $\alpha \in (0,1]$. 
The \brev restriction to \erev this type of boundary condition has been motivated by \cite{SD11} and \cite{SZ08}, which present numerical analysis for this class of boundary conditions on bounded domains in the context of various electromagnetic phenomena. In \brev a slightly different \erev form, this class of nonlinearities further appears in the evolution boundary condition studied in \cite{SV12}. Well-posedness results and analysis for Maxwell's equations with such boundary conditions can be found in \cite{ELN02}.
\brev As a consequence of the rescaling of $\H$, the boundary condition described here differs from the boundary conditions studied in the mentioned literature by the positive factor $\mu$. \erev

\subsection{Contributions of this paper}

The present paper gives, to the best of the author's knowledge, the first numerical analysis of time-dependent electromagnetic scattering with a nonlinear boundary condition. The derivation of the nonlinear time-dependent boundary integral equation results from the combination of techniques presented \brev for the acoustic scattering setting \erev in \cite{BanjaiRieder},\cite{BanjaiLubich2019} and \cite{BLN20}, with the electromagnetic Calderón operator proposed in \cite{KL17}. As such, the present boundary integral equations can be understood as a generalization of the case of linear boundary conditions for electromagnetic scattering, which was presented in \cite{NKL22}.

While the formulation of the boundary integral equation \brev arises from applying established techniques to the present \erev problem, the stability and error analysis conducted in this paper builds on novel techniques and yields new results. A particular challenge is the power-law form of the nonlinear boundary condition, which does not fulfill a monotonicity condition \brev as strong as required for \erev the numerical analysis in the acoustic case \cite{BanjaiLubich2019}. The present error analysis 
includes the following new ideas.
\begin{itemize}
	\item Energy techniques based on time-discrete transmission problems, where errors and defects are rewritten through Green's formula in terms of discrete fields away from the boundary, were derived and utilized.
	\item The stability analysis \brev is conducted \erev in the presence of a weak monotonicity condition fulfilled by the nonlinearity (as provided by Lemma~\ref{lem:pointwise-bounds-a}). This difficulty is circumvented by an a priori estimate on the numerical solution and a series of Hölder inequalities.
	\item A new discrete partial integration inequality for Runge-Kutta convolution quadrature discretizations based on Radau IIA multistage methods is shown and utilized.
	\item Time-harmonic bounds with superior dependence on the Laplace parameter $s$ of the potential operators for the time-harmonic Maxwell's equations, in the context of the functional analytic setting of the nonlinearity, are shown and employed to obtain pointwise error bounds \brev for \erev the numerical solution away from the boundary.
\end{itemize}



Combining these techniques yields error bounds with explicit convergence rates under regularity assumptions on the exact solution. Finally, the \brev present paper contains, to the best of the author's knowledge, the first numerical experiments for the presented problem \eqref{maxwell-tot-EB}--\eqref{power-law}.\erev


\brev 
\subsection{Structure of the paper}

The mathematical content of this paper starts from the next section, which provides the functional analytic framework surrounding Maxwell's equations and the nonlinear boundary condition. 
 Time-dependent nonlinear boundary integral equations are derived in Section~$3$ and a continuous stability result is formulated and proven. Section~$4$ introduces the Runge-Kutta convolution quadrature and gives a time-discrete scheme. 

The error analysis is then conducted in Section~$5$, in which rate-specific error bounds are shown for the numerical solution on the boundary $\Gamma$ and in the domain $\Omega$. 
Finally, Section~6 presents numerical experiments. Convergence plots give empirical error rates and visualize an example simulation of a scattered wave for a given scatterer.
\erev
\section{Framework and analytical background}
\label{section:SettingsGIBC}

%

Let  $\left(\E^\text{inc},\H^\text{inc}\right)$ denote \brev incident \erev waves, solutions to the time-dependent Maxwell's equations on the complete space $\R^3$, with initial support in the exterior domain $\Omega$ away from the boundary $\Gamma$.  
The quantities of interest are the scattered fields $\E=\E^\text{tot}-\E^\text{inc}$ and $\H=\H^\text{tot}-\H^\text{inc}$, which solve the following initial--boundary value problem of Maxwell's equations:
\begin{alignat}{2}
	\,\pt \E - \curl \H &= 0 \quad &&\text{in} \quad\Omega, \label{MW1}\\ 
	 \,\pt \H + \curl \E &=0 \quad &&\text{in} \quad\Omega,\label{MW2}\\
	\E\times \bnu +\a(\H\times \bnu+\H\inc\times \bnu
	) \times \bnu &=-\Einc \times \bnu \quad &&\text{on }\Gamma. \label{eq:gibc}
\end{alignat}
As the initial support of the \brev incident \erev waves is away from the boundary, the initial values in $\Omega$ for both $\E$ and $\H$ vanish. 

Asymptotic conditions for $|\bx|\to\infty$ are not necessary, as the finite wave speed $c=1$ implies that the fields $(\E,\H)$ have bounded support at any time~$t$. 

\brev Deriving well-posedness results for this nonlinear scattering problem are beyond the scope of this paper and can be found in the literature, c.f. \cite{ELN02}. \erev


 For the derivation of a weak formulation \brev and, subsequently, of \erev boundary integral equations it is crucial to give a functional analytic framework that is appropriate for the nonlinear scattering problem. 
The functional analytic setting of the nonlinear boundary condition \eqref{eq:gibc} needs to reconcile the properties of the tangential trace $\gamma_T$ and the nonlinear operator \brev spawned by \erev the composition with the nonlinearity $\ba$. The following section starts with the description of appropriate spaces for $\gamma_T$.
\subsection{Tangential trace and trace space $\mXG$ }
For a continuous vector field defined on the closure of the exterior domain, $\bv:\overline\Omega\to\C^3$, 
we define the \emph{ tangential trace} 
\begin{align*}
\gamma_T \bv= \bv|_\Gamma \times \bnu \qquad\text{on }\Gamma,
\end{align*}
where $\bnu$ denotes the outer unit surface normal.

Green's formula for the $\curl$ operator yields, for sufficiently regular vector fields $\bu,\bv:\overline \Omega\ \rightarrow  \mathbb C^3$, the identity
\begin{equation}\label{eq: Green}
\int_\Omega \curl \bu \cdot \bv- \bu \cdot \curl \bv \,   \textrm{d} \bx = \int_\Gamma (\gamma_T \bu \times \bnu) \cdot \gamma_T \bv \,\textrm{d}\bx,
\end{equation}
where the Euclidean inner product on $\C^3$ is denoted by the dot $\cdot$, defined by $\a \cdot \b = \overline{\a}^\top \b$ for $\a,\b\in\C^3$. The skew-hermitian sesquilinear form on the right-hand side, also referred to as the \emph{anti-symmetric pairing}, is denoted for continuous tangential vector fields on the boundary $\bphi,\bpsi:\Gamma\to\C^3$ by
\begin{equation}\label{skew}                               
[\bphi,\bpsi]_\Gamma = \int_\Gamma (\bphi\times\bnu)\cdot \bpsi \, \textrm{d}\sigma.
\end{equation}

Plugging solutions of Maxwell's equations into the Green's formula, i.e. setting ${\bu = \E}$ and ${\bv = \H}$ yields, for the exterior domain $\Omega$, the identity
\begin{align}\label{eq:time-green}
\begin{split}
\left[\gamma_T \H , \gamma_T \E \right]_\Gamma &=
\int_\Omega \curl \H\cdot \E   - \H \cdot \curl \E \,  \textrm{d} \bx
  \\ &=
 \dfrac{1}{2} \pt \int_\Omega  \abs{\E}^2 +  \abs{\H}^2\textrm{d} \bx.
 \end{split}
\end{align}


In the following, we describe the functional analytic setting of the tangential trace $\gamma_T$ as it has been derived in \cite{AV96} for smooth domains and in \cite{BCS02} for Lipschitz domains. \brev A natural space for the electromagnetic fields $\E$ and $\H$ is \erev
\begin{align*}
{\H(\curl,\Omega) = \{\bv \in \bL^2(\Omega)\,:\, \curl \bv \in \bL^2(\Omega) \}}.
\end{align*}

The tangential trace $\gamma_T$ extends to a surjective bounded linear operator from this space into a trace space $\mXG$, namely
$$
\gamma_T \, \colon \, \H(\curl,\Omega) \rightarrow \mXG.
$$

The Hilbert space $\mXG$, \brev equipped \erev with the appropriate norm $\|\cdot\|_\mXG$, is occasionally referred to as the \emph{proper trace space}.
More background on the functional analytical setting for $\gamma_T$ is described in the surveys \cite[Sect.\,2.2]{BH03} and \cite[Sect.\,5.4]{Ned01}.

The proper trace space $\mXG$ consists of those functions in the tangential subspace of the Sobolev space $\H^{-1/2}(\Gamma)$  with surface divergence in $H^{-1/2}(\Gamma)$, i.e.
\begin{align*}
\mXG =\{\bphi \in \H^{-1/2}(\Gamma)\,:\, \bphi\times\bnu = 0, \, \,  \text{div}_\Gamma \, \bphi \in H^{-1/2}(\Gamma) \}.
\end{align*}
For a precise \brev definition \erev of the Sobolev space $H^{-1/2}(\Gamma)$ (in particular for Lipschitz domains) we refer to the cited publications above. 

 The anti-symmetric pairing $[ \cdot,\cdot ]_\Gamma$ extends to a non-degenerate continuous sesquilinear form on $\mXG\times \mXG$. As a consequence, $\mXG$ becomes its own dual by installing $[\cdot,\cdot]_\Gamma$ as the anti-duality.

The treatment of the nonlinear generalized impedance boundary conditions requires the combination of the natural setting for Maxwell's equations \brev given above with an appropriate setting for the power-law nonlinearity $\ba$. \erev The subsequent section describes some basic properties of $\ba$ and gives a \brev suitable \erev functional analytic framework of the associated nonlinear operator.

\subsection{Functional analytical setting for \brev the power-law nonlinearity $\ba$\erev}
Let $\alpha \in(0,1]$ be a given constant, then we repeat the power-law type nonlinearity from the boundary condition discussed in \cite{SZ08} and occuring in the nonlinear evolution equation described in \cite{SV12} and write
\begin{align*}
\ba(\bx) = \abs{\bx}^{\alpha -1}\bx, \quad \text{ for } \bx\in\mathbb{R}^3.
\end{align*}
This nonlinearity is positive, in the sense that for any $\bx\in\mathbb R^3$, we trivially obtain
\begin{align}\label{coerc-a}
	\bx\cdot \ba(\bx) = \abs{\bx}^{1+\alpha}.
\end{align}
This identity plays a crucial role in stability estimates regarding both the continuous problem \eqref{MW1}--\eqref{eq:gibc} and the numerical scheme proposed in the subsequent sections. To derive error estimates, a stronger form of positivity is necessary and almost provided by an intermediate result from \cite{SZ08}: the nonlinearity $\ba$ is \emph{monotone}, namely for arbitrary $\bu,\bv\in \mathbb R^3$ it holds that
\begin{align}\label{monotone-a}
	(\bu-\bv)\cdot (\ba(\bu)-\ba(\bv)) \ge 0.
\end{align}

The following Lemma gives bounds which will in particular include a stronger form of the monotonicity \eqref{monotone-a}, crucial in the error analysis in the following sections.

 \begin{lemma}[Pointwise bounds on $\ba$] \label{lem:pointwise-bounds-a}
 	For $\alpha \in(0,1]$ the nonlinearity $\ba(\bx)= \abs{\bx}^{\alpha-1}\bx$ fulfills a positivity condition stronger \brev than monotonicity, \erev namely for arbitrary ${\bu,\bv\in\mathbb R^3}$ it holds that
 	\begin{align}\label{pos:a}
 		(\bu-\bv)\cdot	(\ba(\bu)-\ba(\bv)) \ge \alpha (\abs{\bv}+\abs{\bu})^{\alpha-1}\abs{\bu-\bv}^2.
 	\end{align}
 	Furthermore, $\ba$ is Hölder continuous,
 	as the bound
 	\begin{align}\label{hoelder}
 		\abs{\ba(\bu)-\ba(\bv)}\le 2 \abs{\bu-\bv}^{\alpha},
 	\end{align}
 	holds for all $\bu,\bv\in\mathbb R^3$.
 \end{lemma}
 \begin{proof}
 	The coercivity result builds upon an argument of the proof from \cite[Lemma 2.1]{SV12}, which is repeated here for the convenience of the reader. The Jacobian of the nonlinearity $\ba\,:\, \mathbb R^3\rightarrow \mathbb R^3$, in the following denoted by $\boldsymb D\ba\,:\,\mathbb R^3 \rightarrow \mathbb R^{3\times 3}$, is derived from standard differentiation rules and has the explicit form
 	\begin{align}\label{eq:Da-explicit-form}
 		\boldsymb D\ba(\bx) = (\alpha-1)\abs{\bx}^{\alpha-3} \bx \bx^T + \abs{\bx}^{\alpha-1}\bI_3\ \ \text{ for } \bx\in\mathbb R^3\setminus \{0\}, 
 	\end{align}
 	where $\bI_3\in \mathbb R^{3\times 3}$ denotes the identity matrix.
 	

The Jacobian is positive definite for all $\bx\in \mathbb R^3\setminus \{0\}$, \brev since \erev for any $\by\in \mathbb R^3$ we can apply the Cauchy-Schwarz inequality to obtain the estimate
 	\begin{align*}
 		\by^T \boldsymb D {\ba}(\bx) \by &=  (\alpha-1)\abs{\bx}^{\alpha-3}  (\bx^T\by)^2 + \abs{\bx}^{\alpha-1}\abs{\by}^2
 		\\ & \ge \left((\alpha-1)\abs{\bx}^{\alpha-3}  \abs{\bx}^2 + \abs{\bx}^{\alpha-1} \right)\abs{\by}^2 = \alpha \abs{\bx}^{\alpha-1}\abs{\by}^2.
 	\end{align*}
	 
 	The stated strong monotonicity \eqref{pos:a} of the nonlinearity $\ba$ is now a consequence of the positive definiteness of the Jacobian in combination with the fundamental theorem of calculus, which yields
 	\begin{align*}
 		(\bu-\bv)\cdot(\ba(\bu)-\ba(\bv)) &= (\bu-\bv)\cdot \int_0^1 \boldsymb D\ba(\bv+\theta(\bu-\bv)) (\bu-\bv) \, \mathrm d \theta
 		\\ &\ge
 		\alpha \int_0^1\abs{\bv+\theta(\bu-\bv)}^{\alpha-1} \abs{\bu-\bv}^2 \,\mathrm d \theta
 		\\ &\ge
 		\alpha (\abs{\bv}+\abs{\bu})^{\alpha-1}\abs{\bu-\bv}^2.
 	\end{align*}
 	For a proof of the Hölder continuity we refer the reader to \cite[Lemma 6.4]{SV12}.

 \end{proof}
 
 Pointwise bounds of the nonlinearity imply bounds \brev for \erev the nonlinear operator defined by the composition with $\a$, namely the operator $\bu\mapsto \ba \circ \bu $ for arbitrary ${\bu :\Gamma \rightarrow \mathbb R^3}$.
A natural space for this nonlinear operator is given by the tangential $\bL^p$ space for $p>1$ on the boundary $\Gamma$, which reads 
\begin{align*}
\bL_T^{p}(\Gamma) = \left\{ \bu\in \bL^{p}(\Gamma) \,\, |\,\bu\cdot \bnu=0 \,\right\},
\end{align*}
complete with the associated norm $\norm{\bu}_{\bL^p(\Gamma)}$ inherited from the full space $\bL^p(\Gamma)$.


The following result clarifies the relation of the nonlinear operator $\a$ with the tangential space $\bL_T^p(\Gamma)$.
\begin{lemma}\label{lem:setting-a}
	The nonlinear operator defined by the composition with the nonlinearity $\a(\bx)=\abs{\bx}^{\alpha-1}\bx$ for $\alpha \in (0,1]$ is a well-posed bijective operator 
	\begin{align}\label{a-operator}
		\a: \bL^{1+\alpha}_T(\Gamma) \rightarrow \bL^{\frac{1+\alpha}{\alpha}}_T(\Gamma).
	\end{align}
	
\end{lemma}
 \begin{proof}
 The well-posedness of $\ba$ on the stated spaces follows by observing 
 \begin{align}\label{bound-a}
 	\norm{\ba(\bu)}_{\bL^{\frac{1+\alpha}{\alpha}}(\Gamma)} &=
 	\norm{\abs{\bu}^{\alpha-1}\bu}_{\bL^{\frac{1+\alpha}{\alpha}}(\Gamma)}
 	=\norm{\abs{\bu}^{\alpha}}_{L^{\frac{1+\alpha}{\alpha}}(\Gamma)}
 	= \norm{\bu}^{\alpha}_{\bL^{1+\alpha}(\Gamma)},
 \end{align}
which in particular implies the left-hand side is bounded for any $\bu\in \bL_T^{1+\alpha}(\Gamma)$.
The well-posedness of the operator defined through the composition with the inverse of $\ba$, which has the closed form
 $$\ba^{-1}(\bx) = \abs{\bx}^{\frac{1-\alpha}{\alpha}} \bx,$$ is readily apparent by the same argument, proving that $\ba$ is a bijection.
 \end{proof}

%
%
%

 As the boundary $\Gamma=\partial \Omega$  of the scatterer is a bounded surface, we have the following chain of dense inclusions
 \brev
\begin{align}\label{chain-dense}
\bL^{\frac{1+\alpha}{\alpha}}_T(\Gamma)	\subset \bL^2_T(\Gamma) \subset \bL^{1+\alpha}_T(\Gamma)  .
\end{align}
\erev
Choosing $\bL_T^2$ as the pivot space, these spaces are dual to each other, since the reciprocal of their exponents add to one. 
 More precisely, the $\bL^2$-scalar product, denoted by $(\cdot,\cdot)_\Gamma$, is a continuous hermitian bilinear form on $\bL^{1+\alpha}_T(\Gamma)\times \bL^{\frac{1+\alpha}{\alpha}}_T(\Gamma)$ and makes these spaces their respective dual. The continuity is a consequence of the Hölder inequality, which guarantees for boundary functions $\bu$ and $\bv$ of appropriate regularity 
 the bound
\begin{align}\label{Lp-dual-est-1}
(\bu,\bv)_\Gamma &\le \norm{\bu}_{\bL^{1+\alpha}(\Gamma)}\norm{\bv}_{\bL^{\frac{1+\alpha}{\alpha}}(\Gamma)}
.
\end{align} 
Consequently, when understood in the setting of Lemma~\ref{lem:setting-a}, the nonlinear operator $\ba$ maps into the dual of its domain, where the anti-duality between both spaces is explicitly given by the extension of the $\bL^2$- pairing.

Turning towards the composition of the nonlinearity $\ba$ with traces of solutions to Maxwell's equations, we introduce the dense subspace 
$$\mVG=\mXG\cap \bL_T^{1+\alpha}(\Gamma)\subset\mXG,$$  
 equipped with the norm 
$
	{\norm{\bphi}_{\mVG}=\norm{\bphi}_{\mXG}+\norm{\bphi}_{\bL^{1+\alpha}(\Gamma)}}.
$
\brev
\begin{remark}[The scope of the analysis with respect to the nonlinearity $\ba$]\label{rem:class-a}
	In the subsequent sections, we develop an extensive numerical analysis of the described nonlinear scattering problem with the power-law type nonlinearity \eqref{power-law}. The analysis only relies (with respect to $\ba$) on the properties shown in Lemma~\ref{lem:pointwise-bounds-a}, the funcional analytic setting of Lemma~\ref{lem:setting-a} and the positivity \eqref{pos:a}. The presented analysis therefore extends to the larger class of nonlinearities that fulfill these identities (with an appropriate Banach space setting) and analogously for nonlinearities with similar bounds.
\end{remark}
\erev

\subsection{Temporal Sobolev spaces and convolutions}

\noindent Let $V$ \brev and $W$ be a arbitrary Banach spaces and further let \erev  $L(s)\colon V \rightarrow \brev W\erev$ be an analytic family of bounded linear operators for $\Re s> 0 $. Assume further that $L$ is \textit{polynomially bounded} in the \brev following way: there exists a real $\kappa$ and $\nu\ge 0$, and,  for every $\sigma >0$, \erev there exists a positive constant $M_\sigma <\infty$, such that
\begin{align}\label{eq:pol_bound}
\norm{L(s)}_{W\leftarrow V}&\leq M_\sigma \dfrac{\abs{s}^\kappa}{(\Re s)^{\nu}}, \quad\ \text{ Re } s \ge \sigma>0.
\end{align}

This polynomial bound (in terms of the parameter $s$) ensures that the inverse Laplace transform of $L(s)$ is a distribution of finite order of differentiation, vanishing on the negative real half-line $t < 0$. 
Throughout the paper, we use the operational calculus notation \brev of Heaviside \erev, which reads for any sufficiently regular function $g:[0,T]\to V$ 
\begin{equation} \label{Heaviside}
L(\pt)g = (\mathcal{L}^{-1}L) * g.
\end{equation}
The above definition is therefore a shorthand for the temporal convolution\brev, understood in a generalized distributional sense (c.f. \cite[Equation (2.2)]{L94})\erev, of the inverse Laplace transform of $L$ with~$g$. For two \brev analytic families of operators $A(s)$ and $B(s)$ satisfying \eqref{eq:pol_bound} \erev and mapping into compatible spaces for their composition to be well-posed, the associativity of convolution and the product rule of Laplace transforms yield the composition rule $B(\pt)A(\pt)g = (BA)(\pt)g$. 

Let $V$ denote a Hilbert space and further let $r\in \mathbb R$. We denote the Sobolev space of real order $r$  and $V$-valued functions on $\R$ by $H^r(\R,V)$. Furthermore, we expand this notation to finite intervals $(0, T )$ by writing
$$
H_0^r(0,T;V) = \{g|_{(0,T)} \,:\, g \in H^r(\R,V)\ \text{ with }\ g = 0 \ \text{ on }\ (-\infty,0)\} . 
$$
For integer order $r\ge 0$ the natural norm on $H_0^r(0,T;V)$ is equivalent to the norm $\| \pt^r g \|_{L^2(0,T;V)}$.
\brev The result \cite[Lemma 2.1]{L94} makes the Heaviside notation \eqref{Heaviside} rigorous, by extending the temporal convolution to operators acting on the temporal Sobolev spaces:  \erev 
Let $L(s)$ be an analytic family of polynomially bounded operators in the half-plane $\text{Re }s > 0$. Then, $L(\pt)$ extends by density to a bounded temporal linear operator
\begin{equation}\label{sobolev-bound}
L(\pt) : H^{r+\kappa}_0(0,T;V) \to H^r_0(0,T;W),
\end{equation}
for arbitrary real-valued orders $r$. It should be noted that the inclusion $H^r_0(0,T;W)\subset C^{r-1}([0,T];W)$ further implies pointwise bounds for $r\ge 1$.

\subsection{Weak formulation of the nonlinear boundary condition}

Let $\bphi$ denote an arbitrary continuous tangential vector field  on $\Gamma$. Taking the anti-symmetric product $[\cdot,\cdot]_\Gamma$ of the boundary condition \eqref{eq:gibc} with $\bphi$ yields
\begin{equation}
\label{gibc-weak - pre}
[\bphi,\gamma_T \E]_\Gamma +[\bphi, \ba(\gamma_T \H+\gamma_T \H\inc 
) \times \bnu]_\Gamma =-[\bphi,\gamma_T\E\inc]_\Gamma.
\end{equation}

Noting that the cross product with the unit normal $\bnu$ in the nonlinear term simplifies to a $\bL^2$- product \brev results in \erev the following \emph{weak formulation of the boundary condition \eqref{eq:gibc}}: the tangential traces $\gamma_T \E \in \bL^2(0,T;\mXG)$ and 
$\gamma_T \H \in \bL^2(0,T;\mVG)$, the boundary data of solutions $\E,\H\in \bL^2(0,T;\H(\curl,\Omega))\cap \H^1(0,T;\bL^2(\Omega)^3)$ to the Maxwell's equations in $\Omega$ with zero initial conditions \brev fulfills the weak boundary condition if \erev
\begin{equation}\label{gibc-weak}
[\bphi,\gamma_T \E]_\Gamma + (\bphi,\ba(\gamma_T\H+\gamma_T\H\inc ))_\Gamma = [\gamma_T \E^\text{inc},\bphi]_\Gamma
\qquad\text{for all $\bphi\in\mVG$},
\end{equation}
 for almost every $t\in (0,T)$.
All terms appearing in this formulation are well-defined under the stated regularity assumptions.

\subsection{\brev Laplace-domain \erev Maxwell's equations}
\label{section:time harmonic Maxwell}
Our interest lies in time-dependent problems, the study of which requires time-dependent potential and boundary operators. This section gives a short introduction into time-harmonic operators and their properties in the Laplace domain, whose implications on the time domain will be discussed in subsequent sections.


The \emph{time-harmonic Maxwell's equations} read, for $s\in\mathbb{C}$ with $\Re s>0$,
\begin{alignat}{2}
\label{TH-MW1}
s \widehat{\E}-\curl \widehat{\H} &= 0 
\quad &&\text{in} \ \Omega ,
\\ 
\label{TH-MW2}
s \widehat{\H} + \curl \widehat{\E} &=0 \quad &&\text{in} \ \Omega .
\end{alignat}




We recall basic notations associated with the boundary integral operators for the time-harmonic Maxwell's equations, following \cite{BH03,Ned01}.
The \emph{ fundamental solution} reads
\begin{align*}
G(s,\bx)= \dfrac{e^{-s \abs{\bx}}}{4\pi\abs{\bx}}, \qquad \Re s>0,\  \bx\in\mathbb{R}^3 \setminus \{0\}.
\end{align*}
Let $\bvarphi$ denote a regular complex-valued tangentia vector field on the boundary $\Gamma$. 
The electromagnetic \emph{single layer potential} operator $\mathcal{S}(s)$, applied to $\bvarphi$ and evaluated at $\bx\in \mathbb{R}^3 \setminus \Gamma$, is defined by \ecl
\begin{align*}
\mathbfcal{S}(s)\bvarphi(\bx)= -s\int_\Gamma G(s,\bx-\by)\bvarphi(\by)\, \text{d}\by + s^{-1} \nabla \int_\Gamma G(s,\bx-\by) \divG \bvarphi(\by) \, \text{d}\by,
\end{align*} 
and the electromagnetic \emph{double layer potential} operator $\mathbfcal{D}(s)$ is
defined by
\begin{align*}
\mathbfcal{D}(s)\bvarphi(\bx) = \curl\int_\Gamma G(s,\bx-\by)\bvarphi(\by)\text{d}\by.
\end{align*}

Any outgoing solution to the time-harmonic Maxwell's equations \brev is recovered from its \erev tangential traces by the representation formulas
\begin{align}
\label{eq:time-harmonic-kirchhoff-E}
\widehat{\E} &= - \, \mathbfcal{S}(s)\bigl( \gamma_T \widehat{\H}\bigr) + \mathbfcal{D}(s)\bigl(-\gamma_T\widehat{\E}\bigr)\quad \text{ in } \Omega , \\
\label{eq:time-harmonic-kirchhoff-H}
\widehat{\H} &= - \mathbfcal{D}(s) \bigl( \gamma_T \widehat{\H}\bigr) \,- \, \mathbfcal{S}(s) \bigl(-\gamma_T\widehat{\E}\bigr) \quad \text{ in } \Omega .
\end{align}

Despite our interest in the boundary value problem, which is purely formulated on the exterior domain $\Omega$, it will prove to be useful to employ techniques from the theory of transmission problems. These formulations are posed on the full space $\mathbb R^3$, which is assumed to \brev be partitioned as \erev ${\mathbb R^3 = \Omega^-\bigcupdot\Gamma \bigcupdot \Omega^+}$.
Whenever the framework of transmission problems is employed, $\Omega^+$ denotes the exterior domain of interest (elsewhere referred to as $\Omega$), whereas $\Omega^-$ denotes the bounded interior domain.

 Quantities defined purely on the exterior domain, such as the unkown scattered fields, are naturally extended by zero in the inside of the scatterer $\Omega^-$.

 We define \emph{jumps} and \emph{averages}, which are bounded operators from ${\H(\curl,\mathbb R^3\setminus \Gamma)}$ into the trace space $\mXG$, by
\begin{align*}
\jmp{\gamma_T}=\gamma^+_T-\gamma^-_T,\quad\quad
\avg{\gamma_T}=\tfrac12\left(\gamma^+_T+\gamma^-_T\right).
\end{align*}

Using the average trace operator we define the electromagnetic \emph{single and double layer boundary operators} as the composition with the potential operators
$$
\bV(s)=\avg{\gamma_T}\circ \mathbfcal{S}(s), \qquad \bK(s)=\avg{\gamma_T}\circ \mathbfcal{D}(s).
$$ 

Building on the boundary operators, we define the \emph{Calder\'{o}n operator} as introduced in \cite{KL17}, with a sign corrected in \cite{NKL2020}: 
\begin{align*}
\Cald(s)=\begin{pmatrix}
-\bV(s) & \bK(s) \\
-\bK(s) & -\bV(s)
\end{pmatrix}
= \avg{\gamma_T}\circ\begin{pmatrix}
-\mathbfcal{S}(s) & \mathbfcal{D}(s) \\
-\mathbfcal{D}(s) & -\mathbfcal{S}(s)
\end{pmatrix}.
\end{align*}
The jump relations of the boundary integral operators now imply the following central identity.

\brev Any solution to the time-harmonic Maxwell's equations $\widehat \E,\widehat \H \in \H(\curl,\mathbb R^3 \setminus \Gamma)$ fulfills \erev

\begin{align}\label{eq:calderon-jump}
\Cald(s)\begin{pmatrix}\jmp{\gamma_T} \widehat \H\\ -\jmp{\gamma_T} \widehat \E \end{pmatrix}=
\begin{pmatrix}
\avg{\gamma_T}\widehat{\E} \\ 
\avg{\gamma_T}\widehat{\H}
\end{pmatrix} .
\end{align}

The notation associated to the skew-hermitian pairing $[\cdot,\cdot]_\Gamma$ is extended from $\mXG\times \mXG$ to $\mXG^2\times \mXG^2$ componentwise:
$$
\left[ \begin{pmatrix} \bvarphi \\ \bpsi \end{pmatrix} , \begin{pmatrix} \bxi \\ \boldsymb\eta \end{pmatrix} \right]_\Gamma = 
[ \bvarphi,\bxi]_\Gamma + [ \bpsi,\boldsymb\eta]_\Gamma.
$$

The electromagnetic potential and boundary operators extend to bounded operators on their respective spaces, with the \brev trace space $\mXG$ on the boundary and $\H(\curl,\Omega)$ in the domain $\Omega$\erev. Crucial for an analysis of the time-dependent analogues of these operators are bounds where the $s-$dependence of the constants is explicitly known, effectively demanding polynomial bounds of the type \eqref{eq:pol_bound}. 
Estimates explicit in $s$ have been derived in \cite[Theorem 4.4]{BBSV13} and were sharpened in \cite[Lemma 3.4]{NKL22}. 
\\

\begin{lemma}\cite[Lemma 3.2 and Lemma 3.8]{NKL22}\label{lem:laplace-bounds}
The electromagnetic single and double-layer potential operators $\mathbfcal S(s)$ and $\mathbfcal D(s)$ extend to bounded linear operators from $\mXG$ to $\H(\curl,\mathbb R^3 \setminus \Gamma)$, which are bounded for $\Re s >0$ by
	\begin{align*}
	&\| \mathbfcal S(s) \|_{\H(\curl,\mathbb R^3\setminus \Gamma) \leftarrow \mXG} \le \Ctrace \dfrac{\abs{s}^{2}+1}{\Re s},
	\quad\
	\| \mathbfcal D(s) \|_{\H(\curl,\mathbb R^3\setminus \Gamma) \leftarrow \mXG} \le \Ctrace \dfrac{\abs{s}^{2}+1}{\Re s}.
	\end{align*}
\end{lemma}
Point evaluations \brev of the potentials \erev away from the boundary fulfill time-harmonic bounds of the same structure with an additional factor, which exponentially decays with respect to the real part of $s$. Bounds of this type are shown in \cite{BBSV13} for smooth domains and in \cite{NKL22} for Lipschitz domains. In the context of this work, slightly different norms are of interest, namely operator norms which derive from the $\bL^p(\Gamma)$ space. The following \brev lemma \erev gives such bounds, for surfaces that are at least $C^2$.  The higher regularity assumption on the boundary stems from the necessity of a partial integration formula on the boundary, to treat the divergence in the second summand of the single layer potential operator.
\begin{lemma}\label{lem:pointwise-potential-bound}
\brev Assume $\Gamma = \partial \Omega$ to be at least $C^2$ and let $\bx\in\Omega$ be a point away from the boundary, with  distance $d=\dist(\bx,\Gamma)>0 $. \erev There exists a positive constant $C$ independent of $s$ and $\bvarphi$, such that the pointwise bounds
 \begin{align*}
 \abs{\left(\mathbfcal S (s)\bvarphi\right) (\bx)} &\le C \abs{s}e^{-d \Re s} \norm{\bvarphi}_{\bL^p(\Gamma)}, \\
 \abs{\left(\mathbfcal D (s)\bvarphi\right) (\bx)} &\le C \abs{s} e^{-d \Re s} \norm{\bvarphi}_{\bL^p(\Gamma)},
 \end{align*}
 holds for all $1\le p < \infty$. By densitiy, \brev point evaluations of the \erev potential operators extend to linear bounded operators of the type $\mathbfcal S_{\bx} (s) : \bL^p(\Gamma)\rightarrow \mathbb C^3$, fulfilling the bound above. 
\end{lemma}
\begin{proof}
The bounds are a direct consequence of Hölder's inequality. To show the bound for the second integral of the single-layer potential operator, we additionally use a partial integration on the surface, \brev which is provided by \cite[Lemma 16.1]{GT98}), and yields in the present setting \erev
\begin{align*}
\int_\Gamma G(s,\bx-\by) \divG \bvarphi(\by) \, \text{d}\by 
= 
-\int_\Gamma \left(\nabla_{\Gamma}G(s,\bx-\by)\right) \bvarphi(\by) \, \text{d}\by.
\end{align*}
\end{proof}

\section{Maxwell's equations with nonlinear boundary conditions}


This section combines the framework of temporal Sobolev spaces and convolutions with the electromagnetic time-harmonic operators, to enable a treatment of the time-dependent nonlinear scattering problem. A fundamental building block of our analysis is the time-dependent representation formula, which reads:

Let ${\E,\H \in \bL^2(0,T;\H(\curl,\Omega))}$ be time-dependent Maxwell solutions with vanishing initial conditions associated with their 
\brev tangential traces \erev $\bvarphi,\bpsi\in \bL^2(0,T;\mXG)$ defined via
\begin{equation}\label{phi-psi}
{\bvarphi=\gamma_T\H}, \qquad {\bpsi=-\gamma_T \E}.
\end{equation}
The electromagnetic fields $(\E,\H)$ and their respective boundary data $(\bvarphi,\bpsi)$ then fulfill the time-dependent representation formulas
\begin{align}
\label{eq:time-dependent-representation-E}
\E & = - \mathbfcal{S}(\pt)\bvarphi+\mathbfcal{D}(\pt)\bpsi, \\
\label{eq:time-dependent-representation-H}
\H & = - \mathbfcal{S}(\pt)  \bvarphi - \mathbfcal{D}(\pt) \bpsi . 
\end{align}

The time-dependent analogues to the boundary integral operators and consequently the Calderón operator are defined via the Heaviside notation of operational calculus \brev \eqref{Heaviside} \erev. Applying the inverse Laplace transform and the convolution theorem to \eqref{eq:calderon-jump} yields the \brev jump relations \erev of the time-dependent Calderón operator 
\begin{align}\label{eq:cald-identity}
\Cald(\pt)
\begin{pmatrix}
\gamma_T \H \\
-\gamma_T \E
\end{pmatrix}
= \dfrac{1}{2}
\begin{pmatrix}
\gamma_T\E \\
\gamma_T \H
\end{pmatrix} .
\end{align}
\brev Note that for the above expression to hold, we implicitly extended $\E$ and $\H$ by zero in the interior domain. \erev
These \brev jump relations \erev \brev have particularly been used \erev in the treatment of several linear and nonlinear boundary conditions in the context of hyperbolic problems \cite{BanjaiRieder,BanjaiLubich2019,BLN20,NKL22}. Following the combined ideas of these previous papers, we start by adding a symmetric block operator on both sides and arrive at
\begin{align}\label{eq:Bimp_s}
\begin{split}
\Caldimp(\pt)
\begin{pmatrix}
\gamma_T \H \\
-\gamma_T \E
\end{pmatrix}
=
\begin{pmatrix}
\gamma_T \E \\
0
\end{pmatrix},
\quad \quad 
\Caldimp(\pt) = \Cald(\pt)+ 
\begin{pmatrix}
0 & -\tfrac{1}{2}\bI \\
-\tfrac{1}{2}\bI  & 0
\end{pmatrix}
\end{split}.
\end{align}
Testing both sides with $(\boldsymb\eta,\boldsymb\xi)\in \mVG \times\mXG$ yields 
\begin{align*}
\left[
\begin{pmatrix}
\boldsymb\eta \\\boldsymb \xi
\end{pmatrix},
\Caldimp(\pt)
\begin{pmatrix}
\bvarphi \\
\bpsi
\end{pmatrix}
\right]_\Gamma
=\left[
\boldsymb\eta,\gamma_T \E\right]_\Gamma.
\end{align*}
Inserting the weak formulation of the nonlinear boundary condition \eqref{gibc-weak} on the right-hand side and rearranging all unknown terms to the left-hand side yields the weak formulation of the boundary integral equation studied throughout the rest of this paper.

%

\bigskip\noindent
{\bf Boundary integral equation:}
{\it Find, \brev for all $t \in[0,T]$, the boundary densities $(\bvarphi(t), \bpsi(t))\in \mVG\times\mXG$ \erev such that, for all ${(\boldsymb\eta,\boldsymb\xi) \in \mVG\times\mXG}$ it holds that}
\begin{align}
\label{bie-weak-nonlinear}
\left[
\begin{pmatrix}
\boldsymb\eta \\\boldsymb \xi
\end{pmatrix},
\Caldimp(\pt)
\begin{pmatrix}
\bvarphi \\
\bpsi
\end{pmatrix}
\right]_\Gamma+
(\bseta, \a(\bvarphi+\gamma_T\boldsymb\H\inc ))_\Gamma = [\gamma_T \E^\text{inc},\boldsymb\eta]_\Gamma.
\end{align}
Solutions of the time-dependent boundary integral equation coincide with Maxwell's solution fulfilling the weak form of the nonlinear boundary condition \eqref{gibc-weak}, which is proved in the next section.

\begin{remark}\label{rem:shift}
	The boundary integral equation above can be shifted in the frequency domain in the following sense. Let $\sigma >0$ be some constant and let $\widetilde \Caldimp(s) =  \Caldimp(s+\sigma)$. Then, \eqref{bie-weak-nonlinear} is equivalent to the boundary integral equation
	\begin{align}
	\label{bie-weak-shifted}
	\left[
	\begin{pmatrix}
	\boldsymb\eta \\ \boldsymb\xi
	\end{pmatrix},
	\widetilde \Caldimp(\pt)
	\begin{pmatrix}
	\widetilde	\bvarphi \\
	\widetilde	\bpsi
	\end{pmatrix}
	\right]_\Gamma+e^{-\sigma t}
	(\boldsymb\eta, \ba(e^{\sigma t}\widetilde \bvarphi+\gamma_T\boldsymb\H\inc ))_\Gamma = e^{-\sigma t}[\gamma_T \E^\text{inc},\boldsymb\eta]_\Gamma,
	\end{align}
	where the boundary densities are shifted via $\widetilde \bvarphi = e^{-\sigma t}\bvarphi$ and $\widetilde \bpsi = e^{-\sigma t}\bpsi$. Although this boundary integral equation is equivalent to \eqref{bie-weak-nonlinear}, their numerical discretization differ. In particular, parts of the subsequent error analysis only holds for the discretization of this shifted boundary integral equation, though numerical experiments indicate that the shift is not necessary for practical computations \brev (see Figure~\ref{fig:shifts}). The occurence of artificial shifts, that are unnecessary for practical computations but necessary for the convergence analysis, also arises in the context of space-time Galerkin methods for time-dependent boundary integral equations \cite{BH86a}.  \erev
\end{remark}

\subsection{\brev Estimates on the solution of the boundary integral equation \erev}

To prepare our investigations into the stability of the time-dependent nonlinear boundary integral equations, we introduce the following time-dependent transmission problem. \brev This result states \erev the central property of the potential operators and rigorously associates a transmission problem to any pair of time-dependent densities in $\mXG$ with sufficient temporal regularity.
 
 	Let $(\bvarphi,\bpsi)\in \H^k_0(0,T;\mXG\times \mXG)$ denote boundary densities, which are not \brev necessarily \erev boundary data of solutions to the time-dependent Maxwell's equations. The representation formulas \eqref{eq:time-dependent-representation-E}--\eqref{eq:time-dependent-representation-H} define fields $\E,\H \in \H^{k-2}_0(0,T;\H(\curl,\mathbb R ^3\setminus \Gamma))$, which are solutions to the time-dependent transmission problem 
 	\begin{alignat}{3}
 	\quad& \pt\E-\curl \H &&=0 
 	\quad\quad\quad &&\text{in} \quad \mathbb{R}^3\setminus\Gamma ,\label{eq:transmis-1}
 	\\
 	\quad& \pt\H + \curl \E &&=0 \quad &&\text{in} \quad \mathbb{R}^3\setminus\Gamma ,\label{eq:transmis-2} \\
 	&\quad\quad\quad \,\jmp{\gamma_T}\H&&=\bvarphi\,,\label{eq:transmis-3} \\
 	&\quad\quad-\jmp{\gamma_T}  \E&&=\bpsi \,.\label{eq:transmis-4}
 	\end{alignat}
 \brev The quantities $\partial_t \E$ and $\partial_t \H$ are to be understood as elements in the temporal Sobolev space $\H^{k-3}_0(0,T;\H(\curl,\mathbb R ^3\setminus \Gamma))$ and the identities above should therefore be understood in the same way\erev.
The time-dependent Maxwell's equations \eqref{eq:transmis-1}-\eqref{eq:transmis-2} hold by construction of the potential operators, whereas \eqref{eq:transmis-3}--\eqref{eq:transmis-4} are consequences of the \brev jump relations \erev of the potential operators.

Applying this result to solutions of the boundary integral equation gives the following theorem, a stability \brev result for \erev the boundary integral equation, which bounds solutions $(\bvarphi,\bpsi)$ and their respective fields $(\E,\H)$ in terms of the incident fields. Solutions of the nonlinear scattering problem \eqref{MW1}--\eqref{eq:gibc} and the boundary integral equation are shown to coincide. Consequently, we obtain a stability and uniqueness result for solutions of the nonlinear scattering problem, however, proving existence of solutions is beyond the scope of this paper. Well-posedness results are found in \cite{ELN02}.

 
\begin{theorem}(Bounds on the solution of the nonlinear scattering problem)\label{proposition: well-posedness-time-dependent}

Let   ${(\bvarphi,\bpsi)\in \H^{3}(0,T;\mVG\times\mXG)}$ be a solution of the nonlinear boundary integral equation \eqref{bie-weak-nonlinear}. The solution pair $(\bvarphi,\bpsi)$ is, in the sense of \eqref{phi-psi}, the boundary data of unique electromagnetic fields $\E$ and $\H$, which are strong solutions of Maxwell's equations and fulfill the weak formulation of the boundary condition \eqref{gibc-weak}. 
	
	Furthermore, there exists a constant $C>0$ such that the following bound on the boundary densities holds
	\begin{align*}
		&\int_0^{T}
		\norm{\bvarphi}_{\bL^{1+\alpha}(\Gamma)}^{1+\alpha}
		+\norm{\bpsi}_{\bL^{\frac{1+\alpha}{\alpha}}(\Gamma)}^{\frac{1+\alpha}{\alpha}}
		\mathrm d t 
	 \le
		C \int_0^{T} \norm{\gamma_T \H^{\mathrm{inc}}}^{1+\alpha}_{\bL^{1+\alpha}(\Gamma)}+\norm{\gamma_T \E^{\mathrm{inc}}}^{\frac{1+\alpha}{\alpha}}_{\bL^{\frac{1+\alpha}{\alpha}}(\Gamma)} \mathrm dt,
	\end{align*}
	under the assumption that all terms on the right-hand side are finite.
	Additionally, the pointwise (in time) norms of the electromagnetic fields $\E,\H\in \bC(0,T;\bL^2(\Omega))$ are bounded by the same estimate, namely for all $t\in[0,T]$ we have
	\begin{align*}
		\norm{\E(t)}_{\bL^2(\Omega)}^2+\norm{\H(t)}_{\bL^2(\Omega)}^2
		\le 	C \int_0^{t} \norm{\gamma_T \H^{\mathrm{inc}}}^{1+\alpha}_{\bL^{1+\alpha}(\Gamma)}+\norm{\gamma_T \E^{\mathrm{inc}}}^{\frac{1+\alpha}{\alpha}}_{\bL^{\frac{1+\alpha}{\alpha}}(\Gamma)} \mathrm d t'.
	\end{align*}
\brev In both cases, the constant $C>0$ depends only on $\alpha$. \erev
\end{theorem}
\begin{proof}

The uniqueness is a direct consequence of the monotonicity  of $\ba$ and \brev properties of the time-domain Calderón operator $\Cald(\pt)$, which are \erev transported from the Laplace domain \cite[Lemma 3.1]{KL17} to the time domain via \cite[Lemma 2.2]{BLS15}.

We continue with the connection of that solution to the nonlinear scattering problem. Let $\E,\H\in \H^{1}_0(0,T;\H(\curl,\mathbb R^3\setminus \Gamma))$ be solutions to the associated transmission problem \eqref{eq:transmis-1}--\eqref{eq:transmis-4} defined by the time-dependent representation formulas. 

The \brev jump relations \erev of the temporal Calderón operator imply, \brev in the context of the transmission problem \eqref{eq:transmis-1}--\eqref{eq:transmis-4}, the identity \erev
\begin{align}
	\nonumber
	\Caldimp(\pt)\begin{pmatrix}\bvarphi\\ \bpsi\end{pmatrix}
	&=
	\Cald(\pt)\begin{pmatrix}\bvarphi\\ \bpsi\end{pmatrix}
	- \dfrac{1}{2}
	\begin{pmatrix}\bpsi\\ \bvarphi\end{pmatrix}
	\\ &
	=
	\begin{pmatrix}
		\avg{\gamma_T\E} \\ \avg{\gamma_T\H}
	\end{pmatrix}
	-\dfrac{1}{2}
	\begin{pmatrix}-\jmp{\gamma_T \E} \\ \jmp{\gamma_T  \H}\end{pmatrix}
	=
	\begin{pmatrix}
		\gamma_T^+\E\\
		\gamma_T^-\H
	\end{pmatrix}.
	\label{eq:Bimp-reduction}
\end{align}
In particular, both the trace theorem and the time-harmonic bounds on the Calderón operator in combination with \cite[Lemma 2.1]{L94} show that $\gamma_T^+\widehat{\E} $ and $\gamma_T^-\widehat{\H}$ are elements in $\H_0^{1}(0,T;\mXG)$.
Setting $\bseta$ and $\bxi$ pairwise to zero reduces the weak formulation \eqref{bie-weak-nonlinear}  to 
\begin{alignat}{2}
	\left[
	\boldsymb\eta,\gamma_T^+\E
	\right]_\Gamma+
	(\boldsymb\eta, \ba(\bvarphi+\gamma_T \H^{\text{inc}}
	))_\Gamma &= [\gamma_T \E^\text{inc},\boldsymb\eta]_\Gamma,  \quad &&\text{ for all }  \boldsymb\eta \in \mVG,\label{eq:cont-transformed-formulation-1}
	\\ [\bxi,\gamma_T^-\H ]_\Gamma &= 0, &&\text{ for all }  \bxi \in \mXG.\label{eq:cont-transformed-formulation-2}
\end{alignat}
The second equation implies $\gamma_T^-\H = 0$ for all $t\in[0,T]$, which plugged into the integrated Green's formula \eqref{eq:time-green} yields
\begin{align*}
	\dfrac{1}{2}\int_{\Omega^-}   \abs{\E}^2 +  \abs{\H}^2\textrm{d} x &=-\pt^{-1}\left[\gamma^-_T \H ,\gamma^-_T \E  \right]_\Gamma = 0.
\end{align*}
Consequently, the electromagnetic waves $\E$ and $\H$ vanish in the inner domain $\Omega^-$ and the boundary densities are given by the boundary data of the outer fields, namely \eqref{eq:transmis-3}-\eqref{eq:transmis-4} imply ${\bvarphi= \gamma^+_T \H}$ and ${\bpsi= -\gamma^+_T\E}$. 
Plugging these identites into \eqref{eq:cont-transformed-formulation-1} shows that $\E$ and $\H$, restricted to the outer domain $\Omega^+$, are the desired fields which fulfill the nonlinear boundary condition \eqref{eq:gibc}.

We turn our attention to the stated bounds, starting with Green's formula \eqref{eq:time-green} on the exterior domain $\Omega^+$, which reads in terms of the boundary densities
\begin{align*}
	\dfrac{1}{2}\ \left( \norm{\E}_{\bL^2(\Omega^+)}^2 +  \norm{\H}_{\bL^2(\Omega^+)}^2\right) &=\pt^{-1}\left[ \gamma^+_T \H,\gamma^+_T \E  \right]_\Gamma  =\pt^{-1}\left[ \bvarphi,\gamma^+_T \E  \right]_\Gamma .
\end{align*}
 By testing \eqref{eq:cont-transformed-formulation-1} with $\bvarphi$ and integrating both sides, we arrive at the term on the right-hand side, in which we insert Green's formula as described above to arrive at
\begin{align*}
	\dfrac{1}{2}\ \left( \norm{\E(t)}_{\bL^2(\Omega^+)}^2 +  \norm{\H(t)}_{\bL^2(\Omega^+)}^2\right) +
\int_0^t(\bvarphi, \ba(\bvarphi+\gamma_T \H^{\text{inc}}
))_\Gamma \,\mathrm d t' &= \int_0^t[\gamma_T \E^\text{inc},\bvarphi]_\Gamma \,\mathrm dt',
\end{align*}
for all $t\in[0,T]$.

Continuing with the nonlinear term on the left-hand side, we introduce an intermediate term, which is subsequently estimated via the Hölder inequality and the bound \eqref{bound-a} on $\ba$, which yields
\begin{align*}
	(\bvarphi, \,&\a(\bvarphi+\gamma_T \H^{\text{inc}}
	))_\Gamma 
	=
	 \norm{\bvarphi+\gamma_T \H^{\text{inc}}
		}_{\bL^{1+\alpha}(\Gamma)}^{1+\alpha}
	-(\gamma_T\H^{\text{inc}}, \a(\bvarphi+\gamma_T \H^{\text{inc}}
	))_\Gamma
	\\ & \ge
	\norm{\bvarphi+\gamma_T \H^{\text{inc}}}_{\bL^{1+\alpha}(\Gamma)}^{1+\alpha}
	-\left( \norm{\gamma_T\H^{\text{inc}}}_{\bL^{1+\alpha}(\Gamma)} \norm{ \bvarphi+\gamma_T \H^{\text{inc}}}_{\bL^{1+\alpha}(\Gamma)}^\alpha\right)
	\\ & \ge
\norm{\bvarphi+\gamma_T \H^{\text{inc}}
	}_{\bL^{1+\alpha}(\Gamma)}^{1+\alpha}
-\left(C \norm{\gamma_T\H^{\text{inc}}}_{\bL^{1+\alpha}(\Gamma)}^{1+\alpha} +\dfrac{1}{2}\norm{\bvarphi +\gamma_T \H^{\text{inc}}
}_{\bL^{1+\alpha}(\Gamma)}^{1+\alpha}\right),
\end{align*}
where the final estimate is obtained by the generalized Young's inequality.
Applying the same dual estimate on the resulting right-hand side again and absorbing the term depending on $\bvarphi$ leads to 
\begin{align*}
	&\norm{\E}_{\bL^2(\Omega^+)}^2 +  \norm{\H}_{\bL^2(\Omega^+)}^2 +\int_0^T \norm{\bvarphi}^{1+\alpha}_{\bL^{1+\alpha}(\Gamma)} \mathrm d t 
	\\
	& \quad \quad \le 
	C\int_0^T \norm{\gamma_T\E^{\text{inc}}}_{\bL^{\frac{1+\alpha}{\alpha}(\Gamma)}}^{\frac{1+\alpha}{\alpha}}
	+\norm{\gamma_T\H^{\text{inc}}}_{\bL^{1+\alpha}(\Gamma)}^{1+\alpha} \mathrm d t.
\end{align*}
It remains to show the stated bound on $\bpsi$, which is obtained by inserting ${\bpsi= -\gamma^+_T\E}$ into \eqref{eq:cont-transformed-formulation-1}, which yields for arbitrary $\bseta\in\mVG$ the estimate
\begin{align*}
\left[\boldsymb\eta,\bpsi	\right]_\Gamma
& =
	(\boldsymb\eta, \a(\bvarphi+\gamma_T \H^{\text{inc}}))_\Gamma
	- [\gamma_T \E^\text{inc},\boldsymb\eta]_\Gamma
	\\ 
	& \le
	\norm{\boldsymb\eta}_{\bL^{1+\alpha}(\Gamma)}\left(\norm{\bvarphi+\gamma_T \H^{\text{inc}}}_{\bL^{1+\alpha}(\Gamma)}^{\alpha}
	+\norm{\gamma_T \E^{\text{inc}}}_{\bL^{\frac{1+\alpha}{\alpha}(\Gamma)}}\right),
\end{align*}
where the Hölder inequality has been used to estimate both summands.
The density of $\mVG \subset \bL^{1+\alpha}(\Gamma)$ finally implies
\begin{align*}
	\norm{\bpsi}_{\bL^{\frac{1+\alpha}{\alpha}}(\Gamma)} &=
	\sup_{\boldsymb \eta \in \mVG}\dfrac{\left[
	\boldsymb \eta,\bpsi
	\right]_\Gamma}{\norm{\boldsymb \eta}_{\bL^{1+\alpha}(\Gamma)}} \le \norm{ \bvarphi+\gamma_T \H^{\text{inc}}}_{\bL^{1+\alpha}(\Gamma)}^{\alpha}
+\norm{\gamma_T \E^{\text{inc}}}_{\bL^{\frac{1+\alpha}{\alpha}(\Gamma)}}.
\end{align*}
Taking both sides to the power of $\frac{1+\alpha}{\alpha}$ yields the stated result.
\end{proof}
\brev
\begin{remark}[On the assumptions of Theorem~\ref{proposition: well-posedness-time-dependent}]
The additional regularity, required for the boundary densities $(\bvarphi,\bpsi)$, stems from the necessity to use point evaluations of the corresponding electromagnetic fields $(\E,\H)$. The resulting bounds indicate that such strong assumptions are not necessary, which could be formalized by a density argument in the context of linear problems. For the nonlinear setting here, such an argument might be possible, but is beyond the scope of this paper.
\end{remark}
\erev

\section{Semi-discretization in time by Runge--Kutta convolution quadrature}
\label{section:time semi-discrete}

\subsection{Runge-Kutta convolution quadrature}

%
A Runge-Kutta method  with $m$-stages is uniquely determined by its coefficients, which are collected in the \emph{Butcher tableau}
\begin{equation*}
\mathscr{A} = (a_{ij})_{i,j = 1}^m , \quad \b = (b_1,\dotsc,b_m)^T,
\quad \text{and} \quad \c = (c_1,\dotsc,c_m)^T .
\end{equation*}
%
The stability function of the Runge--Kutta method is given by $R(z) = 1 + z b^T (I - z \mathscr{A})^{-1} \bone$, where $\bone = (1,1,\dotsc,1)^T \in \R^m$. \brev The Runge--Kutta method is said to be A-stable if $I - z \mathscr{A}$ is nonsingular and the stability function fulfills $\abs{R(z)}\le 1$ for all $z\in\mathbb C$ with $\Re z \le 0$\erev. For more details about Runge-Kutta methods we refer to \cite{HairerWannerII}.

Runge--Kutta methods have been used extensively to construct convolution quadrature methods c.f. \cite{LubichOstermann_RKcq,BanjaiLubich2011,BLM11,BanjaiMessnerSchanz,BanjaiKachanovska,BanjaiRieder,BanjaiLubich2019}. Crucial for the treatment of wave propagation and scattering problems is the A-stability of the underlying time \brev stepping method, which prohibits the use of multistep methods of larger orders than $2$. \erev

A-stable Runge-Kutta methods of arbitrary order yield effective convolution quadrature schemes, which often outperform their counterparts based on multistep methods \cite{B10}.

The \emph{Runge--Kutta differentiation symbol} is defined by:
\begin{equation}
\label{eq:Delta}
\Delta(\zeta) = \Bigl(\mathscr{A}+\frac\zeta{1-\zeta}\bone \b^T\Bigr)^{-1} \in \C^{m \times m}, \qquad 
\zeta\in\C \hbox{ with } |\zeta|<1.
\end{equation}
This expression is well-defined for $|\zeta|<1$ if 
$R(\infty)=1-\b^T\mathscr{A}^{-1}\bone$ satisfies $|R(\infty)|\le 1$. 
In fact, the Sherman--Morrison formula yields for RadauIIA methods (see, e.g., \cite[Section~IV.5]{HairerWannerII})
\begin{align*}
	\Delta(\zeta) &= \mathscr{A}^{-1} -\frac{\zeta}{1-R(\infty)\zeta} \mathscr{A}^{-1} \bone \b^T \mathscr{A}^{-1}
= \mathscr{A}^{-1}(\bI_m -\zeta \bone \be_m^T),
\end{align*}
with $\be_m^T = (0,...,1)\in \mathbb R^m$ and $\bI_m\in \mathbb R^{m\times m}$ denoting the identity matrix.
The Runge--Kutta convolution quadrature weights are operators \brev $\textbf{W}_n(L):V^m \to W^m$ \erev defined by formally replacing the argument $s$ in $L(s)$ by the Runge--Kutta differentiation symbol $\Delta(\zeta)/\tau$, and \brev then expanding the operator-valued matrix function into the power series \erev
\begin{equation*}
\bL\Bigl(\frac{\Delta(\zeta)}\tau \Bigr) = \sum_{n=0}^\infty \textbf{W}_n(L) \zeta^n.
\end{equation*}
The convolution quadrature approximation of the temporal operator $L(\pt)$ is then \brev defined \erev by the discrete convolution
\begin{equation*}
\bigl(L(\partial_t^\tau) \boldsymbol g \bigr)_n = \sum_{j=0}^n \textbf{W}_{n-j}(L) \boldsymbol g_j
\end{equation*}
for any sequence $\boldsymbol g = (\boldsymbol g_n)_{n\in \mathbb N}\in V^m$.

The sequences often arise from function values, for which we introduce the following notation. Let ${ g:[0,T]\to V}$ be a time-dependent function, then we denote the vector of the evaluations at the stages by \brev $\boldsymb g_n = g(\vtn) =  \bigl(g(t_n+c_i\tau)\bigr)_{i=1}^m$. \erev

Generally, we will associate sequences with functions whenever notationally convenient, where sequences \brev that depend on $\tau$ are labeled by an \erev additional superscript $\tau$.

In particular, if $c_m = 1$, as is the case with Radau IIA methods \cite[Section~IV.5]{HairerWannerII}, then
the continuous convolution at $t_{n+1}$ is approximated by the last component of the discrete block convolution:
\begin{equation*}
\bigl(L(\partial_t) g \bigr)(t_{n+1}) \approx \mathbf{e}_m^T \bigl(L( \partial_t^\tau) \boldsymbol g \bigr)_n,
\end{equation*}
where $\mathbf{e}_m = (0, \dotsc, 0, 1)^T \in \R^m$ is the $m$-th unit vector.

The following convolution quadrature approximation result from \cite[Theorem~3]{BLM11}, formulated for the stages of the Radau IIA method in \cite[Theorem 4.2]{BanjaiLubich2019}, yields efficient bounds for temporal defects originating from the employed time discretization.

\begin{proposition}{\cite[Theorem~3]{BLM11}}
	\label{th:RK-CQ}
	\, Let $L$ satisfy \eqref{eq:pol_bound} and consider the Runge--Kutta convolution quadrature based on the Radau IIA method with $m \ge 2$ stages. Let $r>\max(m+1+\kappa,m+1)$ and $g \in C^r([0,T],V)$ satisfy $g(0)=g'(0)=...=g^{(r-1)}(0)=0$. Then, there exists a $\tau_0>0$, such that for $0<\tau\le \tau_0$ and $t_n=n\tau\in[0,T]$ the following error bound holds:
	\begin{align*}
	&\norm{\left(L(\pt^\tau) \boldsymbol g\right)_n-L(\pt)g(\vtn)}_W\\
	&\quad\quad\quad \le
	C \tau^{\min(m+1,m+1-\kappa+\nu)}
	\left(\norm{g^{(r)}(0)}_V+\int_0^t\norm{g^{(r+1)}(\lambda)}_V\mathrm{d}\lambda
	\right),
	\end{align*}
	\brev where $g(\vtn) = \left(g(t_{n}+c_i \tau)\right)_{i=1}^m$. \erev
\end{proposition}
This section \brev transfers \erev results from the time continuous domain onto the time discrete regime. A useful tool for that purpose is the use of generating functions, which we introduce in the following. 
Let $\Phi_n\in V^m$ denote a sequence with finite support and $m$ components in the Banach space $V$. Let furthermore $\sigma>0$ be a real, constant value and $\rho = e^{-\sigma \tau}$ a weight which \brev converges to $1$ \erev for $N\rightarrow \infty$. Operating on the complex contour ${\mathbb S_\rho=\{ z\in\mathbb C \,| \, \abs{z}=\rho\} }$, the generating function is denoted by an additional hat  ${\widehat \Phi : \mathbb S_\rho \rightarrow V^m}$  and defined by 
\brev the image of the $\mathcal Z$ transform applied to the sequence $\Phi$, which reads \erev
\begin{align*}
 \widehat \Phi \,\colon \, \zeta \mapsto \sum_{n=0}^{\infty}\Phi_n \, \zeta^n.
\end{align*}

Bilinear forms are extended to $\V^m\times \V^m$ by weighting with the diagonal weight matrix ${\mathscr B = \text{diag}(\b_1,...,\b_m)}$, which yields for the scalar product $\cdot$ \brev belonging to \erev $\V$ the extended definition
 \begin{align*}
\bu \cdot \bv={(\bu,\bv)_{\b} = \overline{\bu}^T\mathscr B \bv}= \sum_{i=1}^m \b_i \bu_i\cdot \bv_i \quad \quad \bu ,\bv \in \V^m,
 \end{align*}
where $\cdot$ on the right-hand side denotes the underlying \brev complex-valued \erev dot product on $\V$. In the same way, we extend the skew symmetric pairing \eqref{skew}. With \brev respect to \erev this positive bilinear form, the following result holds. 
\begin{lemma}\label{lem:discrete-partial-integration}(Discrete partial integration)
	Let $(\bu_n)_{i=1}^m$ and $(\bv_n)_{i=1}^m$ be vector-valued sequences in $\mathbb R^m$ and consider the $m-$ stage RadauIIA Runge-Kutta method. For any $\epsilon>0$, there exists a positive constant $C$ independent of $\tau$,$\bu$ and $\bv$ such that the following estimate holds
	\begin{align*}
		\sum_{n=0}^\infty \rho^n \bu_n\cdot\bv_n
		\le  \sum_{n=0}^\infty \epsilon \rho^n  \abs{\left((\pt^{\tau})^{-1} \bu\right)_n}^2+C\rho^n\abs{(\pt^{\tau}\bv)_n}^2.
	\end{align*}
	Here, $\pt^\tau$ and $\left(\pt^\tau\right)^{-1}$ denote the convolution quadrature discretization of the temporal convolutions $L(\pt)$ corresponding to $L(s)=s$ and $L(s)=s^{-1}$, respectively.
\end{lemma}
\begin{proof}
	
\brev Let $\mathbb S_{\rho}$ denote the complex circle which is centered around the origin with radius $\rho<1$, .\erev

We start by applying Parseval's theorem to the left-hand side of the stated bound, which yields for all positive $\widetilde\epsilon$
\begin{align*}
\sum_{n=0}^\infty \rho^n \left(\bu^n,\bv^n\right)_{\b} 
&=
 \int_{\mathbb S_{\rho}}
 \left(\widehat \bu,\widehat \bv\right)_{\b} \,\mathrm d \zeta
 \\ &=
\int_{\mathbb S_{\rho}}
 \left(\overline{\Delta(\zeta)^{-1}\widehat \bu }\right)^T  \left(\Delta(\overline{\zeta})^T\mathscr B \Delta(\zeta)^{-1}\Delta(\zeta)\widehat \bv\right) \,\mathrm d \zeta
 \\ &\le 
\int_{\mathbb S_{\rho}}
  \abs{\Delta(\overline{\zeta})^T\mathscr B \Delta(\zeta)^{-1}}\left(
  \dfrac{\widetilde{\epsilon}}{2}\abs{\tau \Delta(\zeta)^{-1}\widehat \bu }^2+ \dfrac{1}{2\widetilde{\epsilon}}\abs{\dfrac{\Delta(\zeta)}{\tau}\widehat \bv}^2\right) \,\mathrm d \zeta,
\end{align*}
where the last inequality holds for all ${\widetilde \epsilon>0}$.
The matrix in the additional factor consisting of the matrix norm is bounded by applying the Sherman-Morrison formula and the triangle inequality, which yields
\begin{align*}
\abs{\Delta(\overline{\zeta})^T\mathscr B \Delta(\zeta)^{-1}}
&= 
\abs{\left( \bI_m- \overline{\zeta} \,\be_m \,\bone^T\right)\mathscr A^{-T}\mathscr B \left(\mathscr{A}+\frac\zeta{1-\zeta}\bone \b^T\right)}
\\ &\le
(1+m^{1/2})\abs{\mathscr A^{-T}\mathscr B\mathscr A}+ \dfrac{\abs{\left(\bI_m - \overline{\zeta} \,\be_m \,\bone^T \right)\mathscr A^{-T}\b \b^T }}{\abs{1-\zeta}}.
\end{align*}
The first summand is independent of $\zeta$ and bounded. To estimate the second summand, we exploit $\c_m=1$ and consequently ${\mathscr A^{-T}\b = \be_m}$, which implies 
\begin{align*}
	\dfrac{\abs{\left(\bI_m - \overline{\zeta} \,\be_m \,\bone^T \right)\mathscr A^{-T}\b \b^T }}{\abs{1-\zeta}}
	=
	 	\dfrac{\abs{\left(\be_m - \overline{\zeta} \,\be_m \,\bone^T \be_m\right) \b^T }}{\abs{1-\zeta}}
	&=
	\dfrac{\abs{\left(\be_m - \overline{\zeta} \,\be_m \right) \b^T }}{\abs{1-\zeta}}
	=
\abs{\be_m \b^T }.
\end{align*}
The given statement is finally deduced by applying the estimate in the inequality above to obtain a constant $C$ only depending on $m$ and $\epsilon$, such that
\begin{align*}
	\sum_{n=0}^\infty \rho^n \left(\bu^n,\bv^n\right)_{\b} 
	&\le 
   \int_{\mathbb S_\rho}
  \epsilon\abs{\tau \Delta(\zeta)^{-1}\widehat \bu }^2+ C\abs{\dfrac{\Delta(\zeta)}{\tau}\widehat \bv}^2 \mathrm d \theta
  \\&=
\sum_{n=0}^\infty \epsilon \rho^n  \abs{\left((\pt^{\tau})^{-1} \bu\right)_n}^2+C\rho^n\abs{(\pt^{\tau}\bv)_n}^2.
\end{align*}

\end{proof}
\brev
\begin{remark}\label{rem:other-RK-partial-integration}
	In view of Proposition~\ref{th:RK-CQ}, we use the convolution quadrature method based on the $m$-stage Radau IIA as the time discretization throughout the paper. The previous lemma generalizes to the more general class of stiffly accurate A-stable Runge--Kutta methods. For example, this includes the Lobatto IIIC method, which has stage order $m-1$ and full classical order $2m-2$ (details are found in \cite{HairerWannerII}). 
\end{remark}
\erev

The convolution quadrature scheme based on RaudauIIA- methods preserves central properties of the temporal operator $L(\pt)$, which is of particular importance to establish a discrete coercivity property, as described in the dedicated paper \cite[Theorem 3.1]{BanjaiLubich2019}.
The following Lemma restates this result for the standard discrete differential operator $\pt^\tau$, which is sufficient for the subsequent error analysis.
\begin{lemma}\label{lem:time-discrete-coercivity}
	\brev Let $V$ now be a Hilbert space with the norm $\norm{\cdot}_V$ and further let $\left\langle\cdot,\cdot\right\rangle$ denote the corresponding inner product, which is extended to $V^m$ with the weight matrix $\mathscr B$.
	\erev The convolution quadrature discretization based on the two-stage Radau IIA method fulfills, for arbitrary $\widetilde N \in\mathbb N$, the estimate
	\begin{align*}
	\tau \sum_{n=0}^N e^{-2  n \tau/T}\left\langle f_n
	,(\pt^\tau f)_n 
	\right\rangle\ge \dfrac{\tau}{2T}
	\sum_{n=0}^N e^{-2  n \tau / T} \norm{f_n}_V^2,
	\end{align*}
	for every sequence $(f_n)_{n\ge 0}$ in $V$. Furthermore, for $m>2$ stages the left-hand side remains positive, namely
	\begin{align*}
	 \sum_{n=0}^N e^{-2  n \tau/T}\left\langle f_n
	,(\pt^\tau f)_n 
	\right\rangle\ge 0.
	\end{align*}
\end{lemma}

 \subsection{Auxiliary result: Time-discrete transmission problem}

 The following Lemma describes a discrete variant of the continuous transmission problem \eqref{eq:transmis-1}--\eqref{eq:transmis-2} and relates an arbitrary sequence in the trace space $\mXG^2$ with the solution to a corresponding transmission problem.
 \begin{lemma}\label{lem:transmis-discrete}
 \brev	Let $\bvarphi = (\bvarphi_n)_{n\ge 0}$ and $\bpsi = (\bpsi_n)_{n\ge 0}$ denote sequences with elements in $ \mXG^m$. We define fields $(\E^\tau,\H^\tau) = \left(\E_n^\tau,\H_n^\tau\right)_{n\ge 0}$, which are sequences in the space $\H(\curl,\mathbb R^3\setminus \Gamma)^m$, by the discrete representation formulas via \erev
 	\begin{align}
 	\E^\tau & = - \mathbfcal{S}(\pt^\tau)\bvarphi+\mathbfcal D(\pt^\tau)\bpsi,\label{eq:discrete-repr-E} \\
 	\H^\tau & = - \mathbfcal D(\pt^\tau)  \bvarphi - \mathbfcal S(\pt^\tau)  \bpsi . \label{eq:discrete-repr-H}
 	\end{align}
 	These fields are exact solutions to the following discrete transmission problem:
 	\begin{alignat}{2}
 	\,\pt^\tau \E^\tau - \curl \H^\tau &= 0 \quad &&\text{in} \quad\mathbb R^3 \setminus \Gamma, \label{MW1-tau-R3-Gamma}\\ 
 	\,\pt^\tau \H^\tau + \curl \E^\tau &=0 \quad &&\text{in} \quad\mathbb R^3 \setminus \Gamma,\label{MW2-tau-R3-Gamma}\\
 	 \jmp{\gamma_T}\H^\tau&= \bvarphi,&& \label{discrete-transmis-varphi}\\
 	 - \jmp{\gamma_T}\E^\tau&=\bpsi. &&\label{discrete-transmis-psi}
 	\end{alignat}
 	
 \end{lemma}
\begin{proof}
The generating function of the representation formula yields formulas for the generating functions $ \wE^\tau (\zeta)$ and $\wH^\tau(\zeta)$, which are of the form
\begin{align*}
\wE^\tau (\zeta) = -\mathbfcal S\left(\dfrac{\Delta(\zeta)}{\tau}\right) \widehat \bvarphi (\zeta )+ \mathbfcal D\left(\dfrac{\Delta(\zeta)}{\tau}\right)\widehat \bpsi (\zeta ).
\end{align*}
The construction and jump relations of the time-harmonic potential operators $\mathbfcal S(s)$ and $\mathbfcal D(s)$ yield the following time-harmonic transmission problem for the generating functions
\begin{alignat*}{2}
\dfrac{\Delta(\zeta)}{\tau} \widehat\E^\tau(\zeta) - \curl \widehat\H^\tau(\zeta) &= 0 \quad &&\text{in} \quad\mathbb R^3\setminus \Gamma,\\ 
\dfrac{\Delta(\zeta)}{\tau} \widehat\H^\tau(\zeta) + \curl \widehat\E^\tau(\zeta) &=0 \quad &&\text{in} \quad\mathbb R^3\setminus \Gamma,\\
\jmp{\gamma_T}\widehat\H^\tau(\zeta)&=\widehat \bvarphi(\zeta) && \text{on} \quad \Gamma,\\
-\jmp{\gamma_T}\widehat\E^\tau(\zeta)&=\widehat \bpsi(\zeta)\quad  && \text{on} \quad \Gamma.
\end{alignat*}
Comparing the coefficients now yields the result as stated.
\end{proof}
The well posedness result of Proposition~\ref{proposition: well-posedness-time-dependent} was enabled through the \brev jump relations \erev of $\Cald(\pt)$, which is a natural property to \brev carry over \erev to the time-discrete Calderón operator $\Cald(\pt^\tau)$. The following Lemma provides such a statement and proves it, again by making use of generating functions and time-harmonic identities.
\begin{lemma}\label{lem: discrete-Calderon-properties}
 Let $\Cald(\pt^\tau)$ denote the convolution quadrature approximation of the time-dependent Calderón operator $\Cald(\pt)$, based on the Radau IIA method with $m$-stages. \brev Let furthermore $\bvarphi = (\bvarphi_n)_{n\ge 0}$ and ${\bpsi = (\bpsi_n)_{n\ge 0}}$ be sequences with elements in $\mXG^m$, with associated discrete fields $\E^\tau = \left(\E^\tau_n\right)_{n\ge 0}$ and $\H^\tau = \left(\H^\tau_n\right)_{n\ge 0}$ with support on $\mathbb R^3\setminus \Gamma$, defined via \eqref{eq:discrete-repr-E}--\eqref{eq:discrete-repr-H}. The Calderón operator then fulfills the jump relations
 \begin{align*}
 	\Cald(\pt^\tau)
 	\begin{pmatrix}
 	\bvarphi \\ \bpsi
 	\end{pmatrix} = 	\begin{pmatrix}
 	\avg{\gamma_T\E^\tau} \\	\avg{\gamma_T\H^\tau}
 	\end{pmatrix}.
 \end{align*}
 \brev Moreover, we have the following identity for all $n\in\mathbb N$\erev
 \begin{align}\label{eq:discrete-Cald-coerc}
 	\left[ 	\begin{pmatrix}
 	\bvarphi_n \\ \bpsi_n
 	\end{pmatrix}
 	,
 	\left(	\Cald(\pt^\tau)
 		\begin{pmatrix}
 		\bvarphi \\ \bpsi
 		\end{pmatrix}\right)_n
 	\right]_\Gamma = 
 	\int_{\mathbb R^3\setminus \Gamma} \E_n^\tau \cdot (\pt^\tau\E^\tau)_n + (\pt^\tau \H^\tau)_n \cdot \H_n^\tau \,\mathrm d\bx.
 \end{align}

\end{lemma}
\begin{proof}
The generating function of the sequence on the left-hand side gives the stated \brev jump relations \erev by employing the time-harmonic \brev jump relations \erev of the Calderón operator, which gives
\begin{align*}
	\widehat{\left(\Cald(\pt^\tau)\begin{pmatrix}
	\bvarphi \\ \bpsi
	\end{pmatrix}\right)}
	&=\Cald\left(\dfrac{\Delta(\zeta)}{\tau}\right)
	\begin{pmatrix}
\bwvarphi(\zeta)\\  \bwpsi (\zeta)
	\end{pmatrix}
	 =
	\begin{pmatrix}
	\avg{\gamma_T\widehat\E^\tau} \\ \avg{\gamma_T\widehat\H^\tau}
	\end{pmatrix}.
\end{align*}	
 Inserting the discrete \brev jump relations \erev into the left-hand side of \eqref{eq:discrete-Cald-coerc} and applying the relations of the boundary densities and their respective fields, as given by \eqref{discrete-transmis-varphi}--\eqref{discrete-transmis-psi}, yields
\begin{align*}
 	\left[ 	\begin{pmatrix}
 	\bvarphi_n \\ \bpsi_n
 	\end{pmatrix}
 	,
 	\left(	\Cald(\pt^\tau)
 	\begin{pmatrix}
 	\bvarphi^\tau \\ \bpsi^\tau
 	\end{pmatrix}\right)_n
 	\right]_\Gamma
 	&=
  	\left[ 	\begin{pmatrix}
  \jmp{\gamma_T} \H^\tau_n \\ -\jmp{\gamma_T} \E^\tau_n
  	\end{pmatrix}
  	,
  	\begin{pmatrix}
  	\avg{\gamma_T}\E_n^\tau \\	\avg{\gamma_T}\H_n^\tau
  	\end{pmatrix}
  	\right]_\Gamma	
  \\ &= 
    [\gamma_T^+\H_n^\tau,\gamma_T^+ \E_n^\tau]_\Gamma - [\gamma_T^-\H_n^\tau,\gamma_T^- \E_n^\tau]_\Gamma.
\end{align*}
Finally, applying Green's formula \eqref{eq: Green} and inserting the discretized Maxwell's equations \eqref{MW1-tau-R3-Gamma}--\eqref{MW2-tau-R3-Gamma} completes the proof by
\begin{align*}
	    [\gamma_T^+\H_n^\tau,\gamma_T^+ \E_n^\tau]_\Gamma - [\gamma_T^-\H_n^\tau,\gamma_T^- \E_n^\tau]_\Gamma 
	    &= \int_{\mathbb R^3\setminus \Gamma} \E_n^\tau \cdot \curl \H_n^\tau-\curl \E_n^\tau \cdot \H_n^\tau \,\mathrm d\bx
	    \\ &= 
	    \int_{\mathbb R^3\setminus \Gamma}  \E_n^\tau \cdot (\pt^\tau\E^\tau)_n + (\pt^\tau \H^\tau)_n \cdot \H_n^\tau \,\mathrm d\bx.
\end{align*}
\end{proof}
\begin{remark}\label{rem:shift-Cald-pos}
 Of particular importance is the combination of this result with Lemma~\ref{lem:time-discrete-coercivity}, which implies a coercivity for $m=2$. For $m>2$ no such result is known, but by introducing a positive shift $\sigma$ and setting $\widetilde\Cald(s) = \Cald(s+\sigma)$, as it appears in the shifted boundary integral equation \eqref{bie-weak-shifted}, one obtains 
 \begin{align*}
 \left[ 	\begin{pmatrix}
 \bvarphi_n \\ \bpsi_n
 \end{pmatrix}
 ,
 \left(	\widetilde \Cald(\pt^\tau)
 \begin{pmatrix}
 \bvarphi \\ \bpsi
 \end{pmatrix}\right)_n
 \right]_\Gamma = 
 \int_{\mathbb R^3\setminus \Gamma} & \E_n^\tau \cdot (\pt^\tau\E^\tau)_n +\sigma\abs{\E^\tau_n}^2  
 \\ +& (\pt^\tau \H^\tau)_n \cdot \H_n^\tau +\sigma\abs{\H^\tau_n}^2 \,\mathrm d\bx,
 \end{align*}
 thus recovering the crucial positivity of the discrete Calderón operator (in combination with the second part of Lemma~\ref{lem:time-discrete-coercivity}). Numerical \brev  experiments indicate that such a shift is not necessary (c.f. Figure~\ref{fig:shifts}), \erev  however the subsequent error analysis depends on such a positivity result. 
 
\end{remark}
\subsection{Convolution quadrature for the nonlinear boundary integral equation}

Discretizing the temporal Calderón operator in the boundary integral equation \eqref{bie-weak-nonlinear} with Runge--Kutta based convolution quadrature yields the following semi-discrete scheme.

\bigskip\noindent
{\bf Time-discrete boundary integral equation:}
{\it Find ${(\bvarphi^n, \bpsi^n)=(\bvarphi^n_i, \bpsi^n_i)_{i=1}^m \in \mVG^m\times\mXG^m}$, such that for all ${({\boldsymb{\eta}},\bxi) \in \mVG\times\mXG}$ and ${n\le N}$ it holds}
\begin{align}
\label{cq-bie-weak}
\left[
\begin{pmatrix}
\boldsymb{\eta} \\ \bxi
\end{pmatrix},
\left(\Caldimp(\pt^\tau)
\begin{pmatrix}
\bvarphi^\tau \\
\bpsi^\tau
\end{pmatrix}\right)_n
\right]_\Gamma+
\left( \boldsymb{\eta} , \a\left(\bvarphi^\tau_n+\gamma_T \H_n^\text{inc} \right) \right)_\Gamma = [\gamma_T \E_n^\text{inc},\boldsymb{\eta}]_\Gamma.
\end{align}
 The numerical solution can then be evaluated by 
\begin{align}
\label{eq:time-dependent-representation-E-tau}
\E^\tau & = - \mathbfcal{S}(\pt^\tau)\bvarphi^\tau+\mathbfcal{D}(\pt^\tau)\bpsi^\tau, \\
\label{eq:time-dependent-representation-H-tau}
\H^\tau & = - \mathbfcal{D}(\pt^\tau) \bvarphi^\tau - \mathbfcal{S}(\pt^\tau)  \bpsi^\tau . 
\end{align}

To limit the extent of the present paper we abstain from giving a convergence analysis of the semi-discretization and directly move on to the space discretization. The proof of error bounds for the full discretization from Theorem~\ref{them:semi-discr-time-experimental} is readily reduced to the semi-discretization and predicts convergence rates of order $m$ in the norms that are stated there.

\section{Full discretization}
We turn our attention to the development and analysis of fully discrete schemes. To achieve this, we start by an introduction of the Galerkin space discretization, suitable for \brev the variational formulation \eqref{cq-bie-weak}. \erev
\subsection{Boundary element method}
Restricting the time-discrete boundary integral equation \eqref{cq-bie-weak} to finite dimensional subspaces $\V_h\subset \mVG$ and $\X_h \subset \mXG$, corresponding to piecewise polynomials defined on a family of triangulations with mesh width $h$, yields the full discretization.
As boundary element spaces, we employ Raviart--Thomas elements of order $k\ge 0$ (c.f. \cite{RT77}) for the discretization of $\V_h$ and $\X_h$ respectively, which are defined on the unit triangle
$\widehat{K}$ as reference element by
\begin{align*}
	\text{RT}_k(\widehat{K}) = \left\{\bx\mapsto \p_1(\bx)+p_2(\bx) \bx \,:\  \p_1\in P_k(\widehat{K})^2\!,\  p_2\in P_k(\widehat{K})\right\},
\end{align*}
where $P_k(\widehat{K})$ is the polynomial space of degree $k$ on $\widehat{K}$.
This definition naturally extends to arbitrary grids by piecewise pull-back to the reference element.

The following approximation result holds for Raviart-Thomas elements and are obtained in this form from the results collected in Lemma 14 and Theorem 15 of \cite{BH03}; see also the original references \cite[Section~III.3.3]{BreF91} and \cite{BC03}.
Here, we use the same notation $\H^p_\times(\Gamma)=\gamma_T \H^{p+1/2}(\Omega)$ for boundary data of higher regularity, as in \cite{BH03}.

\begin{lemma} \label{lem:RT}
	Let $\X_h=\V_h$ be the $k$-th order Raviart--Thomas boundary element space on $\Gamma$. For every $\bxi\in\H^{k+1}_\times(\Gamma)$  the best-approximation error is bounded by
	\begin{align*}
		&\inf_{\bxi_h\in \X_h} \| \bxi_h - \bxi \|_{\mXG}+
h^{1/2} \| \bxi_h - \bxi \|_{\bL^2(\Gamma)} \le C h^{k+3/2} \|\bxi \|_{\H^{k+1}_\times(\Gamma)}.
	\end{align*}
\end{lemma}
The $\bL^{1+\alpha}(\Gamma)$ norm naturally arises in estimates derived from the boundary integral equation (e.g. in Proposition~\ref{prop:phi-bound-tau} and Theorem~\ref{proposition: well-posedness-time-dependent} ). On bounded domains, as is the case for $\Gamma$, this norm is estimated by the $\bL^2(\Gamma)$ norm as a consequence of the Hölder inequality with $q = \frac{2}{1+\alpha}$ and $p = \frac{2}{1-\alpha}$, which implies
\begin{align}\label{est: L1pa-to-L2}
\norm{\bu}^{1+\alpha}_{\bL^{1+\alpha}(\Gamma)} = \int_{\Gamma} \abs{\bu}^{1+\alpha} \mathrm d \bx 
\le
\norm{1}_{L^{\frac{2}{1-\alpha}}(\Gamma)} \norm{\abs{\bu}^{1+\alpha}}_{L^{\frac{2}{1+\alpha}}(\Gamma)} 
= C_{\alpha,\Gamma} \norm{\bu}_{\bL^2(\Gamma)}^{1+\alpha}.
\end{align}

\subsection{Fully discretized boundary integral equation}
The fully discrete scheme of the boundary integral equation \eqref{bie-weak-nonlinear} is now given by combining the convolution quadrature method, used to discretize the temporal Calderón operator, with the boundary element method.

\bigskip\noindent
{\bf Full discretization of the boundary integral equation:} 
{\it Find ${(\bvarphi\tauhn, \bpsi\tauhn) \in (\V_h\times \X_h)^m}$, such that for all $(\bseta_h,\bxi_h) \in \V_h\times \X_h$ and $n\le N$ the following scheme holds}
\begin{align}\label{eq:full-dis}
\left[
\begin{pmatrix}
\bseta_h \\ \bxi_h
\end{pmatrix},
\left(\Caldimp(\pt^\tau)
\begin{pmatrix}
\bvarphi^\tau_h \\
\bpsi^\tau_h
\end{pmatrix}\right)_n
\right]_\Gamma+
\left( \bseta_h , \ba\left(\bvarphi^n_h+\gamma_T \H_n^\text{inc}\right) \right)_\Gamma = [\gamma_T \E_n^\text{inc},\bseta_h]_\Gamma.
\end{align}
The remaining part of this section is devoted to the derivation of error bounds of this scheme. 

Our investigations into the errors of fully discrete solutions start with a stability result, which is desirable on its own but in particular takes a central role in the subsequent derivation of rate specific error bounds.
\subsection{Full discretization: Unconditional bounds on the numerical solution}
The following Proposition bounds the numerical solution in terms of the \brev incident \erev waves, without making any assumptions on the regularity of the exact solution.

\begin{proposition}\label{prop:phi-bound-tau}
	Consider $\bvarphi\tauhn \in \V^m_h$ for all $n\ge 0$, solution to the fully discrete scheme \eqref{eq:full-dis}, where  Radau IIA based Runge-Kutta convolution quadrature with $m$-stages in time and arbitrary boundary element spaces $\V_h$ and $\X_h$ have been employed. Then, the numerical solution is bounded by
	\begin{align*}
	&\sum_{n=0}^N	\norm{\bvarphi\tauhn}_{\bL^{1+\alpha}(\Gamma)}^{1+\alpha} 
	\le C\sum_{n=0}^N\norm{\gamma_T \E^{\text{inc}}_n}_{\bL^{\frac{1+\alpha}{\alpha}}(\Gamma)}^{\frac{1+\alpha}{\alpha}}+	\norm{\gamma_T \Hinc _n}^{1+\alpha}_{\bL^{1+\alpha}(\Gamma)},
	\end{align*}
	where the constant $C$ depends only on $\alpha$.
\end{proposition}
\begin{proof}
	\brev We start by testing the discretized scheme with the numerical solution $(\bvarphi\tauh,\bpsi\tauh)$, which yields after summation\erev
	\begin{align}\label{eq:time-discrete-scheme-tested}
	\begin{split}
	&\sum_{n=0}^N e^{-\sigma n \tau} \left(\left[
	\begin{pmatrix}	\bvarphi\tauhn \\ \bpsi\tauhn \end{pmatrix},
	\left(\Caldimp(\pt^\tau)
	\begin{pmatrix}
	\bvarphi\tauh \\
	\bpsi\tauh
	\end{pmatrix}\right)_n
	\right]_\Gamma+ 
	\left( \bvarphi\tauhn , \ba\left(\bvarphi\tauhn+\gamma_T \H^\text{inc}_n \right) \right)_\Gamma \right)
	 \\ & \quad \quad= \sum_{n=0}^N e^{-\sigma n \tau}[\gamma_T \E^\text{inc}_n,\bvarphi^\tau_n]_\Gamma.
	\end{split}
	\end{align}
	The positivity of the time-discrete operator $\Caldimp(\pt^\tau)$, seen as a direct consequence of  Lemma~\ref{lem:time-discrete-coercivity} applied to the right-hand side of \eqref{eq:discrete-Cald-coerc} \brev where the numerical solution $(\bvarphi\tauhn,\bpsi\tauhn)$ is extended by zero for $n>N$, implies
	\begin{align*}
	\sum_{n=0}^N e^{-\sigma n \tau}\left[
	\begin{pmatrix}	\bvarphi\tauhn \\ \bpsi\tauhn \end{pmatrix},
	\left(\Caldimp(\pt^\tau)
	\begin{pmatrix}
	\bvarphi\tauh \\
	\bpsi\tauh
	\end{pmatrix}
	\right)_n
	\right]_\Gamma \ge 0,
	\end{align*}
. \erev
	The corresponding term in the expression \eqref{eq:time-discrete-scheme-tested} is therefore neglected, to bound the remaining term on the left-hand side by the right-hand side.  Rewriting the summands of the second term yields 
	\begin{align*}
	\left( \bvarphi\tauhn , \ba\left(\bvarphi\tauhn+\gamma_T \H^\text{inc}_n \right) \right)_\Gamma 
	=&
	\norm{\bvarphi\tauhn+\gamma_T \Hinc_n }_{\bL^{1+\alpha}(\Gamma)}^{1+\alpha} 
	-\left( \gamma_T \Hinc_n , \ba\left(\bvarphi\tauhn+\gamma_T \H^\text{inc}_n \right) \right)_\Gamma.
	\end{align*}
	Rearranging gives a constant $C>0$, such that the following intermediate inequality holds
	\begin{align*}
	\sum_{n=0}^N e^{-\sigma n \tau}
	\norm{\bvarphi\tauhn+\gamma_T \H^\text{inc}_n}_{L^{1+\alpha}(\Gamma)}^{1+\alpha}
	\le& C\sum_{n=0}^N e^{-\sigma n \tau}
	[\gamma_T \E^\text{inc}_n,\bvarphi\tauhn]_\Gamma
	\\ &+ C\sum_{n=0}^N e^{-\sigma n \tau} \left( \gamma_T \Hinc_n, \ba\left(\bvarphi\tauhn+\gamma_T \H^\text{inc}_n \right) \right)_\Gamma  .
	\end{align*}
	The stability bound is obtained by estimating the terms on the right-hand side by subsequently applying the \brev Hölder \erev inequality and Young's inequality.
 We start with the first term, which is estimated for all $n\in \mathbb N$ and $\epsilon>0$ by
	\begin{align*}
	[\gamma_T \E^\text{inc}_n,\bvarphi\tauhn]_\Gamma\le\norm{\gamma_T\E^\text{inc}_n}_{\bL^{\frac{1+\alpha}{\alpha}}(\Gamma)}\norm{\bvarphi\tauhn}_{\bL^{1+\alpha}(\Gamma)}
	\le
	C\norm{\gamma_T \E^{\text{inc}}_n}_{\bL^{\frac{1+\alpha}{\alpha}}(\Gamma)}^{\frac{1+\alpha}{\alpha}} +\epsilon \norm{\bvarphi\tauhn}_{\bL^{1+\alpha}(\Gamma)}^{1+\alpha}.
	\end{align*}
		 Choosing $\epsilon$ small enough enables the absorption of the term depending on the numerical solution $\bvarphi\tauhn$.

	The last summand regarding $\bvarphi\tauhn$ is bounded by the same chain of inequalities via
	\begin{align*}
	\left( \gamma_T \Hinc_n , \ba\left(\bvarphi\tauhn+\gamma_T \H^\text{inc}_n \right) \right)_\Gamma 
	 &
	\le C\norm{\gamma_T \Hinc_n}^{1+\alpha}_{\bL^{1+\alpha}(\Gamma)}+\epsilon \norm{\ba\left(\bvarphi\tauhn+\gamma_T \H^\text{inc}_n \right)}^{\frac{1+\alpha}{\alpha}}_{\bL^{\frac{1+\alpha}{\alpha}}(\Gamma)}
	\\ &
	= C\norm{\gamma_T \Hinc_n}^{1+\alpha}_{\bL^{1+\alpha}(\Gamma)}+\epsilon \norm{\bvarphi\tauhn+\gamma_T \H^\text{inc}_n }^{1+\alpha}_{\bL^{1+\alpha}(\Gamma)}.
	\end{align*}
\end{proof}
\subsection{Main result: Error bounds for the full discretization}
 \begin{theorem}\label{them:semi-discr-time-experimental}
  	 Let $(\bvarphi,\bpsi)$ 
  	  be the solution of the boundary integral equation \eqref{bie-weak-nonlinear}, assumed to be of regularity
  	  $$(\bvarphi,\bpsi)\in \H_0^{m+5}\left(0,T;\mXG^2\right)
  	  \cap 
  	  \H_0^{3}\left(0,T;  \H^{k+1}_{\times}(\Gamma)^2\right).$$ Furthermore, let $\gamma_T \H(t)\inc \in \bL^{1+\alpha}(\Gamma)$ for all $t\in[0,T]$.	 
 	Consider the fully discrete boundary densities $(\bvarphi\tauhn,\bpsi\tauhn)\in \V_h^m\times \X_h^m$ for all $n \le N$, solutions to the the fully-discrete boundary integral equations \eqref{eq:full-dis}, discretized by 
 	\begin{itemize}
 		\item Radau IIA based Runge-Kutta convolution quadrature with $m$-stages in time, and
 		\item Raviart-Thomas boundary elements of order $k$ in space.
 	\end{itemize}
 	For $m>2$, we assume the scheme to be applied to the shifted boundary integral equation \eqref{bie-weak-shifted} by some positive shift $\sigma >0$.
 	 The error of the $m$-stage Radau IIA semi-discretization, denoted by ${\be_\bvarphi = \bvarphi_h^{\tau}-\Pi_h\bvarphi}$ and $\be_\bpsi = \bpsi_h^{\tau}-\Pi_h\bpsi$, fulfill the bounds
 	\begin{align*}
 	\left(\tau \sum_{n=0}^N \norm{\left(\left(\pt^{\tau}\right)^{-1} \be_\bvarphi\right)_n}^2_{\mXG}
 	+\norm{\left(\left(\pt^{\tau}\right)^{-1} \be_\bpsi\right)_n}^2_{\mXG}\right)^{1/2}
 	\le C \left(\tau^{m}+h^{\alpha(k+1)} \right).
 	\end{align*}
 Errors in the electromagnetic fields, defined through the discrete representation formulas \eqref{eq:time-dependent-representation-E-tau}--\eqref{eq:time-dependent-representation-H-tau}, are bounded via
  	\begin{align*}
  	\left(\tau \sum_{n=0}^N \norm{\E\tauhn-\E(\vtn)}^2_{\bL^2(\Omega)}
  	+\norm{\H\tauhn-\H(\vtn)}^2_{\bL^2(\Omega)}\right)^{1/2}
  	\le C \left(\tau^{m}+h^{\alpha(k+1)} \right),
  	\end{align*}
  	where the notation $\E(\vtn) = \left(\E(t_n+c_i\tau)\right)_{i=1}^m$ has been applied to the electromagnetic fields. 
  	The constants in the error bounds depend on higher Sobolev norms of the exact solution $(\bvarphi,\bpsi)$, the shift $\sigma$ for $m>2$, the boundary $\Gamma$ and polynomially on the final time $T$.
  
 \end{theorem}

 \begin{proof}
 	Throughout this proof, whenever an expression holds for arbitrary $n\in \mathbb N$ (or respectively at all time points $t_n$), we omit the index for notational convenience. For the sake of presentation, we further assume that $m=2$, thus allowing for the first coercivity property of Lemma~\ref{lem:time-discrete-coercivity}. The proof readily generalizes to $m>2$, by means of Remark~\ref{rem:shift} and the resulting positivity as described in Remark~\ref{rem:shift-Cald-pos}. 
 	
 	We start by inserting a projection of the exact solution, which yields a sequence of  defects 
 	${\bd =(\bd^n)_{n\ge 0}}$, where each of the sequence elements fulfill ${\bd^n\in \V^m_h \times\X^m_h}$ for all $n\ge 0$, such that the perturbed boundary integral equation
 	\begin{align}\label{eq:perturbed-bie}
 	\left[
 	\begin{pmatrix}
 	\bseta_h \\ \bxi_h
 	\end{pmatrix},
 	\Caldimp(\pt^\tau)
 	\begin{pmatrix}
 	\Pi_h \bvarphi \\
 	\Pi_h \bpsi
 	\end{pmatrix}
 	\right]_\Gamma+
 	\left( \bseta_h , \ba\left(\Pi_h\bvarphi+\gamma_T \H^\text{inc} \right) \right)_\Gamma = [\gamma_T \E^\text{inc},\bseta_h]_\Gamma
 	+\left[
 	\begin{pmatrix}
 	\bseta_h \\ \bxi_h
 	\end{pmatrix},\bd
 	\right]_\Gamma
 	\end{align}
 	holds. \brev We define fields associated with the projected boundary densities through the discrete representation formulas \erev
 	\begin{align}\label{def:W}
 	\begin{pmatrix}
 	\E_{\Pi}^{\tau,h} \\ \H_{\Pi}^{\tau,h}
 	\end{pmatrix} & = 
 	\begin{pmatrix}
 	-\mathbfcal{S}(\pt^\tau)\Pi_h\bvarphi+\mathbfcal{D}(\pt^\tau)\Pi_h\bpsi \\
 	- \mathbfcal{D}(\pt^\tau)  \Pi_h\bvarphi - \mathbfcal{S}(\pt^\tau) \Pi_h \bpsi
 	\end{pmatrix}
 	  = \mathbfcal W(\pt^\tau) \begin{pmatrix}
 	  \Pi_h\bvarphi \\ \Pi_h \bpsi
 	  \end{pmatrix} . 
 	\end{align}
 	These intermediate fields approximate the exact fields $\E$ and $\H$ at least in the stated order due to
 	\begin{align*}
 	\begin{pmatrix}
 	\E_{\Pi}^{\tau,h}-\E \\ \H_{\Pi}^{\tau,h}-\H
 	\end{pmatrix}
 	&= (\mathbfcal W(\pt^\tau)-\mathbfcal W (\pt))
 	\begin{pmatrix} \Pi_h\bvarphi \\ \Pi_h \bpsi \end{pmatrix}
 	+ \mathbfcal W(\pt)
 	\begin{pmatrix} \Pi_h\bvarphi-\bvarphi \\ \Pi_h \bpsi -\bpsi \end{pmatrix},
 	\end{align*}
 	which implies, due to the time-harmonic bounds of Lemma~\ref{lem:laplace-bounds} and the general convolution quadrature approximation results of Proposition~\ref{th:RK-CQ}, the existence of a constant $C$ depending only the surface $\Gamma$ and polynomially on the final time $T$, such that for all $n\le N$ we have the bound
 	\begin{align}\label{est: E-H-Pi}
	 	\begin{split}
		 	&\norm{\left(\E_{\Pi}^{\tau,h}\right)_n-\E(\vtn)}_{\H(\curl,\Omega)}+\norm{\left(\H_{\Pi}^{\tau,h}\right)_n-\H(\vtn)}_{\H(\curl,\Omega)}
		 	\\& \quad \le C\left(\tau^m\norm{\begin{pmatrix}\bvarphi \\ \bpsi\end{pmatrix}}_{\H^{m+5}_0(0,T; \mXG^2)} 
		 	+h^{k+3/2}\norm{\begin{pmatrix}\bvarphi \\ \bpsi\end{pmatrix}}_{\H^{3}_0(0,T; \H^{k+1}_{\times}(\Gamma)^2)} \right).
 		\end{split}
 	\end{align} 
 	Subtracting the perturbed scheme from the full discretization yields, by testing with 
 	${\be^n_\bvarphi = \bvarphi\tauhn-\Pi_h\bvarphi(\vtn)}$ and ${\be^n_\bpsi = \bpsi\tauhn-\Pi_h\bpsi(\vtn)}$, the following error equation
 	\begin{align}\label{eq:error-equation}
 	\begin{split}&\left[
 	\begin{pmatrix}
 	\be_\bvarphi^{n} \\ \be_\bpsi^{n}
 	\end{pmatrix},
 	\left(\Caldimp(\pt^\tau)
 	\begin{pmatrix}
 	\be_\bvarphi\\
 	\be_\bpsi
 	\end{pmatrix}\right)_n
 	\right]_\Gamma
 	 +
 	\left( \be_\bvarphi^{n} , \ba\left(\bvarphi^{\tau,h}_n+\gamma_T \H^\text{inc}_n \right)-\ba\left(\Pi_h\bvarphi(\vtn)+\gamma_T \H^\text{inc}_n \right) \right)_\Gamma 
 	\\ & \quad\quad\quad\quad\quad\quad\quad\quad\quad\quad\quad\quad\quad\quad\quad\quad\quad\quad\quad\quad\quad\quad\quad\quad\quad\quad\quad\quad= \left[
 	\begin{pmatrix}
 	\be_\bvarphi^{n} \\ \be_\bpsi^{n}
 	\end{pmatrix},\bd^n
 	\right]_\Gamma.
 	\end{split}
 	\end{align}
 	\brev In the equation above and in the rest of the proof, we use the shorthand notation $\H\inc_n= \H\inc(\vtn)$ for the electromagnetic incident waves. \erev 
 	Note that the term with the nonlinearity $\ba$ is readily estimated from below, by the pointwise monotonicity estimate of $\ba$ from Lemma~\ref{lem:pointwise-bounds-a}.
	The first summand, corresponding to the time-discrete Calderón operator, is estimated from below by \brev applying \eqref{eq:discrete-Cald-coerc} \erev, which yields
	\begin{align*}
\left[
\begin{pmatrix}
\be_\bvarphi^{n} \\ \be_\bpsi^{n}
\end{pmatrix},
\left(\Caldimp(\pt^\tau)
\begin{pmatrix}
\be_\bvarphi\\
\be_\bpsi
\end{pmatrix}\right)_n
\right]_\Gamma
	&=
	\int_{\mathbb R^3 \setminus \Gamma} (\E\tauh-\E_{\Pi}^{\tau,h})_n\cdot \left(\pt^\tau\left(\E\tauh-\E_{\Pi}^{\tau,h}\right)\right)_n
	\\ & \quad\quad\quad \quad+ \left(\H\tauh-\H_{\Pi}^{\tau,h}\right)_n\cdot \left(\pt^\tau\left(\H\tauh-\H_{\Pi}^{\tau,h}\right)\right)_n  \,\mathrm{d\bx}.
	\end{align*}

   To employ the coercivity of the discrete operator $\pt^\tau$, given by Lemma~\ref{lem:time-discrete-coercivity}, a weighted summation on both sides is necessary.
   
    Consequently, summation of the numerical scheme, weighted at the time point $t_n$ with  $\rho^n = e^{-2n\tau/T}$ and inserting this identity into the left-hand side yields a positive constant $C$, such that 
 	\begin{align*}
 	&\sum_{n=0}^{N} \rho^n\left(
 	 \norm{\left(\E\tauh-\E_{\Pi}^{\tau,h}\right)_n}^2_{\bL^2(\mathbb R^3\setminus\Gamma)}+ \norm{\left(\H\tauh-\H_{\Pi}^{\tau,h}\right)_n}^2_{\bL^2(\mathbb R^3\setminus\Gamma)} \right)
 	\\ + &
 	\sum_{n=0}^{N} \rho^n\int_{\Gamma} \left(\abs{\bvarphi\tauhn+\gamma_T \H_n^\text{inc}}+\abs{\Pi_h\bvarphi(\vtn)+\gamma_T \H_n^\text{inc}}
 	\right)^{\alpha -1}\abs{\be^n_\bvarphi}^2 \,\mathrm d \bx
 	\\ & \quad \le C\sum_{n=0}^N  \rho^n\left[\begin{pmatrix}
 	 \be^n_\bvarphi\\  \be^n_\bpsi 
 	\end{pmatrix},\bd^n\right]_\Gamma.
 	\end{align*}
 	
 	We turn towards the estimation of the defect, by subtracting the exact boundary integral equation \eqref{bie-weak-nonlinear} from the perturbed equation \eqref{eq:perturbed-bie} to obtain
 	\begin{alignat*}{2}
 	\left[\begin{pmatrix}
 	 \be_\bvarphi^{n}\\  \be_\bpsi^{n}
 	\end{pmatrix},\bd^n\right]_\Gamma
 	&= 
 	\left[
 	\begin{pmatrix}
 	\be_\bvarphi^{n} \\ \be_\bpsi^{n}
 	\end{pmatrix},
 	\left(\Caldimp(\pt^\tau)
 	\begin{pmatrix}
 	\Pi_h \bvarphi \\
 	\Pi_h \bpsi
 	\end{pmatrix}
 	-\Caldimp(\pt)
 	\begin{pmatrix}
 	\bvarphi \\
 	\bpsi
 	\end{pmatrix}\right)_n
 	\right]_\Gamma \quad \quad && \text{(A)}
 	\\ &+
 	\left(\be_\bvarphi^{n} , \ba\left(\Pi_h\bvarphi(\vtn)+\gamma_T \H_n^\text{inc} \right)-
 	\ba\left(\bvarphi(\vtn)+\gamma_T \H_n^\text{inc} \right) \right)_\Gamma && \text{(B)}.
 	\end{alignat*}
 	The defect has been split into two parts, first into a temporal defect where the approximation of the time-dependent Calderón operator enters and secondly a nonlinear defect depending on the nonlinearity $\ba$. In the following, we estimate these terms successively. \\
 
 	(A) 
 	We start with the temporal defect, for which applying the \brev jump relations \erev of both the discrete and the continuous time-dependent Calderón operator yields
 	\begin{align*}
 	\left[
 \begin{pmatrix}
 \be_\bvarphi \\ \be_\bpsi
 \end{pmatrix},
 \Caldimp(\pt^\tau)
 	\begin{pmatrix}
 	\Pi_h \bvarphi \\
 	\Pi_h \bpsi
 	\end{pmatrix}
 	-\Caldimp(\pt)
 	\begin{pmatrix}
 	\bvarphi \\
 	\bpsi
 	\end{pmatrix}
 	\right]_\Gamma
 	&=
 	\left[
\begin{pmatrix}
\be_\bvarphi \\ \be_\bpsi
\end{pmatrix},
 	\begin{pmatrix}
 	\gamma_T^+ \E_{\Pi}^{\tau,h} -\gamma_T^+ \E \\
 	\gamma_T^- \H_{\Pi}^{\tau,h}
 	\end{pmatrix}
 	\right]_\Gamma.
 	\end{align*} 
 	Writing the numerical solution and the projected exact solution in terms of the jumps of their respective fields and sorting for the inner and outer fields (in the left argument of the duality) yields
 	\begin{align}\label{eq:defect-terms}
 	\begin{split}
 	\left[
 	\begin{pmatrix}
 	\bvarphi_h^\tau-\Pi_h\bvarphi \\ \bpsi_h^\tau-\Pi_h\bpsi
 	\end{pmatrix},
 	\begin{pmatrix}
 	\gamma_T^+ \E_{\Pi}^{\tau,h} -\gamma_T^+ \E \\
 	\gamma_T^- \H_{\Pi}^{\tau,h}
 	\end{pmatrix}
 	\right]_\Gamma &= 
 	\left[
 	\begin{pmatrix}
 	\gamma_T^+(\H\tauh-\H_{\Pi}^{\tau,h}) \\ -\gamma_T^+(\E\tauh-\E_{\Pi}^{\tau,h}) 
 	\end{pmatrix},
 	\begin{pmatrix}
 	\gamma_T^+ \E_{\Pi}^{\tau,h} -\gamma_T^+ \E \\
 	\gamma_T^- \H_{\Pi}^{\tau,h}
 	\end{pmatrix}
 	\right]_\Gamma
 	\\ & -
 	\left[
 	\begin{pmatrix}
 	\gamma_T^-(\H\tauh-\H_{\Pi}^{n,h}) \\ -\gamma_T^-(\E\tauh-\E_{\Pi}^{\tau,h}) 
 	\end{pmatrix},
 	\begin{pmatrix}
 	\gamma_T^+ \E_{\Pi}^{\tau,h} -\gamma_T^+ \E \\
 	\gamma_T^- \H_{\Pi}^{\tau,h}
 	\end{pmatrix}
 	\right]_\Gamma.
 	\end{split}
 	\end{align}
 	These terms, which correspond to errors in the inner and outer domains respectively, are bounded successively starting with the first summand.
 	  Rewriting the only inner trace appearing in the first summand via the discrete transmission problem yields
 	\begin{align*}
 	\gamma_T^-\H_{\Pi}^{\tau,h}&= 	\gamma_T^+\H_{\Pi}^{\tau,h}-\Pi_h\bvarphi
 	=
 	\left(\gamma_T^+\H_{\Pi}^{\tau,h}-\gamma_T^+\H\right)+\left(\bvarphi-\Pi_h\bvarphi\right).
 	\end{align*} 
 	Inserting this identity into the first summand yields
 	\begin{alignat*}{2}
 	&\left[
 	\begin{pmatrix}
 	\gamma_T^+(\H\tauh-\H_{\Pi}^{\tau,h}) \\ -\gamma_T^+(\E\tauh-\E_{\Pi}^{\tau,h}) 
 	\end{pmatrix},
 	\begin{pmatrix}
 	\gamma_T^+ \E_{\Pi}^{\tau,h} -\gamma_T^+ \E \\
 	\gamma_T^- \H_{\Pi}^{\tau,h}
 	\end{pmatrix}
 	\right]_\Gamma 
 	\\
 	&\quad = 
 	\left[\gamma_T^+(\H\tauh-\H_{\Pi}^{\tau,h}),\gamma_T^+ (\E_{\Pi}^{\tau,h} -\E)\right]_\Gamma \quad\quad&&\text{(i)}
 	\\ & \quad-\left[\gamma_T^+(\E\tauh-\E_{\Pi}^{\tau,h})  ,\gamma_T^+\left(\H_{\Pi}^{\tau,h}-\H\right) \right]_\Gamma
 	&&\text{(ii)}
 	\\ &\quad -\left[\gamma_T^+(\E\tauh-\E_{\Pi}^{\tau,h}),
 	\bvarphi-\Pi_h\bvarphi
 	\right]_\Gamma.
 	&&\text{(iii)}
 	\end{alignat*}
 	The next paragraphs are dedicated to the successive estimation of the terms (i)--(iii).
 	
 	(i)
 	Applying Green's formula to the first summand yields
 	\begin{align*}
 	&\left[\gamma_T^+ (\H\tauh-\H_\Pi^{\tau,h})
 	,\, 
 	\gamma_T^+ (\E_{\Pi}^{\tau,h} - \E) 
 	\right]_\Gamma
 	\\&=
 	\int_{\Omega} \curl(\H\tauh-\H_\Pi^{\tau,h} ) \cdot (\E_{\Pi}^{\tau,h}-\E) -(\H\tauh-\H_\Pi^{\tau,h}) \cdot \curl( \E_{\Pi}^{\tau,h} - \E ) \,\mathrm d\bx
 	\\ & \le
 	\int_{\Omega^+} \left(\pt^\tau\E\tauh-\pt^\tau\E_\Pi^{\tau,h} \right)\cdot \left(\E_{\Pi}^{\tau,h}-\E\right) \,\mathrm d\bx
 	\\
 	&+\norm{\H\tauh-\H_\Pi^{\tau,h}}_{\bL^2(\Omega+)} \norm{\curl\E_{\Pi}^{\tau,h} - \curl\E  }_{\bL^2(\Omega+)}.
 	\end{align*}
 	Summation over $n$, 
 	applying the discrete integration bound of Lemma~\ref{lem:discrete-partial-integration} to the first summand and the Cauchy-Schwarz inequality to the second summand consequently implies, for arbitrary $\epsilon>0$ the existence of a positive constant $C$, such that
 	\begin{align*}
 	\sum_{n=0}^{N} &\rho^n \left[\gamma_T^+ \left(\H\tauhn-\left(\H_\Pi^{\tau,h}\right)_n\right)
 	,\, 
 	\gamma_T^+ \left(\left(\E_{\Pi}^{\tau,h}\right)_n - \E(\vtn)\right) \right]_\Gamma
 	\\\le &
 	\sum_{n=0}^{N} \rho^n\left(\epsilon\norm{\left(\E\tauh-\E_\Pi^{\tau,h}\right)_n}^2_{\bL^2(\Omega)} +C \norm{\left(\pt^\tau\E_\Pi^{\tau,h}-\pt^\tau\E\right)_n}^2_{\bL^2(\Omega)}\right)
 	\\  +&
 	\sum_{n=0}^{N}\rho^n\left(\epsilon
 	\norm{\left(\H\tauh-\H_\Pi^{\tau}\right)_n}^2_{\bL^2(\Omega)} + C\norm{\left(\E_\Pi^{\tau,h}-\E\right)_n}^2_{\H(\curl,\Omega)}\right).
 	\end{align*}
 	Choosing $\epsilon>0$ small enough enables the absorption of error terms depending on the numerical solution $\E\tauhn$ and $\H\tauhn$. The error of the intermediate field $\E_\Pi^{\tau,h}$ in the $\H(\curl,\Omega)$ norm is bounded by \eqref{est: E-H-Pi} in the desired order.
 	 The remaining defect term, which is numerically differentiated, is rewritten by exploiting \eqref{MW1-tau-R3-Gamma} and introducing an intermediate term, to obtain
 	\begin{align*}
 	\norm{\left(\pt^\tau\E_\Pi^{\tau,h}-\pt^\tau\E\right)_n}_{\bL^2(\Omega^+)}
 	&\le \norm{\left(\curl\H_\Pi^{\tau,h}-\curl\H\right)_n}_{\bL^2(\Omega^+)}+
 	\norm{\left(\pt \E-\pt^\tau\E\right)_n}_{\bL^2(\Omega^+)}
 	\\ & \le C(\tau^m+h^{k+3/2}),
 	\end{align*}
 	where the bound is the consequence of \eqref{est: E-H-Pi} and Proposition~\ref{th:RK-CQ} respectively.
 	
 	(ii) We repeat the argument structure and again apply Green's formula to obtain
 	\begin{align*}
 	&\left[\gamma_T^+\left(\E\tauh-\E_{\Pi}^{\tau,h}\right),\gamma_T^+\left(\H_{\Pi}^{\tau,h}-\H\right) \right]_\Gamma
 \\	&=
 	\int_{\Omega^+}\curl\left(\E\tauh-\E_{\Pi}^{\tau,h}\right) \cdot  \left(\H_{\Pi}^{\tau,h}-\H\right) -\left(\E\tauh-\E_{\Pi}^{\tau,h}\right)\cdot  \curl\left(\H_{\Pi}^{\tau,h}-\H\right) \mathrm d\bx
 	\\ & \le
 	-\int_{\Omega^+} \left(\pt^\tau\H\tauh-\pt^\tau\H_{\Pi}^{\tau,h}\right)\cdot\left(\H_{\Pi}^{\tau,h}-\H\right) \mathrm d\bx  
 	\\&
 	+\norm{\E\tauh-\E_{\Pi}\tauh}_{\bL^2(\Omega^+)}\norm{\curl\H_{\Pi}\tauh-\curl\H}_{\bL^2(\Omega^+)}   .
 	\end{align*}
 	Applying the discrete integration bound of Lemma~\ref{lem:discrete-partial-integration} to the first summand consequently leads to the estimate
 	\begin{align*}
 	&\sum_{n=0}^N \rho^n	\left[\gamma_T^+\left(\E\tauhn-\left(\E_{\Pi}^{\tau,h}\right)_n\right),\gamma_T^+\left(\left(\H_{\Pi}^{\tau,h}\right)_n-\H(\vtn)\right) \right]_\Gamma
 	\\& \le 
 	\sum_{n=0}^N \rho^n\left(\epsilon\norm{\left(\H\tauh-\H_{\Pi}^{\tau,h}\right)_n}^2_{\bL^2(\Omega^+)} +C\norm{\left(\pt^\tau\H_{\Pi}\tauh-\pt^\tau\H\right)_n}^2_{\bL^2(\Omega^+)}\right)
 	\\ & \quad\quad +
 	\rho^n\left(\epsilon \norm{\left(\E\tauh-\E_{\Pi}^{\tau,h}\right)_n}^2_{\bL^2(\Omega^+)}+ C\norm{\left(\H_{\Pi}^{\tau,h}\right)_n-\H(\vtn)}^2_{\H(\curl,\Omega^+)}
 	 \right),
 	\end{align*}  
 	where $\epsilon>0$ is chosen small enough to absorb the terms depending on the numerical solution.  
 	Applying the discrete identity \eqref{MW2-tau-R3-Gamma} yields further
 	\begin{align*}
\norm{\left(\pt^\tau\H_{\Pi}\tauh-\pt^\tau\H\right)_n}_{\bL^2(\Omega^+)}
 	&\le    	\norm{\curl\left(\E_{\Pi}^{\tau,h}\right)_n-\curl\E(\vtn)}_{\bL^2(\Omega^+)}
 	+
 	\norm{\left(\pt \H-\pt^{\tau}\H\right)_n}_{\bL^2(\Omega^+)}
 	\\ &
 	\le C(\tau^m+h^{k+3/2}).
 	\end{align*} 
 	
 	(iii)  
 	We use the discrete partial integration bound to obtain
 	\begin{align}\label{est:pt-1-E}
 	\begin{split}
 	&\sum_{n=0}^N e^{-n\sigma \tau} \left[\gamma_T^+\left(\E\tauhn-\left(\E_{\Pi}^{\tau,h}\right)_n\right),
 	\bvarphi(\vtn)-\Pi_h\bvarphi(\vtn)
 	\right]_\Gamma
 	\\ & \le
 	\sum_{n=0}^N e^{-n\sigma \tau}\left(\epsilon\norm{\left(\left(\pt^{\tau}\right)^{-1}\gamma_T^+(\E\tauh-\E_{\Pi}^{\tau,h})\right)_n}^2_{\mXG}
    +C\norm{\pt^{\tau}\left(\bvarphi-\Pi_h\bvarphi\right)(\vtn)}^2_{\mXG}\right).
    	\end{split}
 	\end{align}
 	The second summand is of the required order, seen by splitting the discrete time derivative into $\pt^\tau = \pt-(\pt-\pt^\tau)$ and applying Lemma~\ref{th:RK-CQ} to obtain
 	\begin{align*}
 	\norm{\left(\pt^{\tau}-\pt\right)\left(\bvarphi-\Pi_h\bvarphi\right)}_{\mXG}\le C \tau^m \norm{\bvarphi}_{\H_0^{m+3}(0,T;\mXG)}.
 	\end{align*}
 	Applying the trace theorem to the first summand yields
 	\begin{align*}
 	&\norm{\left(\left(\pt^{\tau}\right)^{-1}\gamma_T^+(\E\tauh-\E_{\Pi}^{\tau,h})\right)_n}^2_{\mXG}
 	\le \norm{\left(\left(\pt^{\tau}\right)^{-1}\E\tauh-\left(\pt^{\tau}\right)^{-1}\E_{\Pi}^{\tau,h}\right)_n}^2_{\H(\curl,\Omega^+)}
 	\\&
 	=\int_{\Omega^+}\abs{\left(\left(\pt^{\tau}\right)^{-1}\E_h^\tau-\left(\pt^{\tau}\right)^{-1}\E_{\Pi}^{\tau,h}\right)_n }^2+\abs{\H_h^n-\H_{\Pi}^{n,h}}^2 \,\mathrm d\bx
 	 .
 	\end{align*}
   Furthermore, by \cite[Theorem 3.1]{BanjaiLubich2019} with $L(s)=1$ and $R(s)=s^{-1}$, we estimate the discrete integral in the $\bL^2$-norm with a constant $C$ depending only on the final time $T$, such that
 	\begin{align*}
  \sum_{n=0}^\infty \rho^n\int_{\Omega^+}\abs{\left(\left(\pt^{\tau}\right)^{-1}\E_h^\tau-\left(\pt^{\tau}\right)^{-1}\E_{\Pi}^{\tau,h}\right)_n }^2\,\mathrm d\bx
    \le C \sum_{n=0}^\infty \rho^n\norm{\left(\E\tauh-\E_{\Pi}^{\tau,h} \right)_n}_{\bL^2(\Omega^+)}^2.
 	\end{align*}
  Choosing $\epsilon$ small enough in \eqref{est:pt-1-E} allows for the absorption of the remaining term to the left-hand side. 
 \ \\
 	We turn our attention to the second summand of \eqref{eq:defect-terms}, consisting mostly of defects in term of the traces of the inner domain $\Omega^-$. Structurally, the process of estimation is identical, starting from rewriting the only remaining term depending on outer traces by
 	\begin{align*}
 	\gamma_T^+\E_{\Pi}^{\tau,h}-\gamma_T^+\E&= 	\gamma_T^-\E_{\Pi}^{\tau,h}-\Pi_h\bpsi-\gamma_T^+\E
 	=
 	\left(\gamma_T^-\E_{\Pi}^{\tau,h}-\gamma_T^-\E\right)+\left(\bpsi-\Pi_h\bpsi\right),
 	\end{align*}
 	where we exploited that the exact solution vanishes in the inner domain, i.e. ${\gamma_T^-\E=0}$.
 	Inserting this identity on the right argument of the second summand of \eqref{eq:defect-terms} yields
 	\begin{alignat*}{2}
 	&\left[
 	\begin{pmatrix}
 	\gamma_T^-(\H\tauh-\H_{\Pi}^{\tau,h}) \\ -\gamma_T^-(\E\tauh-\E_{\Pi}^{\tau,h}) 
 	\end{pmatrix},
 	\begin{pmatrix}
 	\gamma_T^+ \E_{\Pi}^{\tau,h} -\gamma_T^+ \E \\
 	\gamma_T^- \H_{\Pi}^{\tau,h}
 	\end{pmatrix}
 	\right]_\Gamma
 	\\&\quad= 
 	\left[\gamma_T^-(\H\tauh-\H_{\Pi}^{\tau,h}),\gamma_T^-\left(\E_{\Pi}^{\tau,h}-\E\right)\right]_\Gamma \quad\quad\quad \quad &&\text{(iv)}
 	\\&\quad +
 	\left[\gamma_T^-(\H\tauh-\H_{\Pi}^{\tau,h}),\bpsi-\Pi_h\bpsi\right]_\Gamma \quad \quad &&\text{(v)}
 	\\ &\quad- 
 	\left[\gamma_T^-(\E\tauh-\E_{\Pi}^{\tau,h}),\gamma_T^- \left(\H_{\Pi}^{\tau,h}-\H\right)
 	\right]_\Gamma. && \text{(vi)}
 	\end{alignat*}
	These terms depending on the inner traces are bounded precisely by the arguments presented to estimate (i)--(iii) respectively.

 	\ \\
 	
 	(B) We introduce the following notation for the nonlinear defect:
 	\begin{align*}
 	\bd_{\ba} =  \ba\left(\Pi_h\bvarphi+\gamma_T \H^\text{inc} \right)-
 	\ba\left(\bvarphi+\gamma_T \H^\text{inc} \right).
 	\end{align*}
 	Rewriting the term of interest (B) by means of a multiplicative intermediate term in combination with Hölder's and Young's inequalities yields
 	\begin{align*}
 	\left(\bvarphi\tauh-\Pi_h\bvarphi , \bd_a\right)_\Gamma
 	& \le 
 	C \int_{\Gamma}\abs{\bvarphi\tauh-\Pi_h\bvarphi}\left(\abs{\bvarphi\tauh+\gamma_T \H^\text{inc}}+\abs{\Pi_h\bvarphi+\gamma_T \H^\text{inc}}
 	\right)^{\frac{\alpha -1}{2}}
 	\\ &
 	\quad\quad \left(\abs{\bvarphi\tauh+\gamma_T \H^\text{inc}}+\abs{\Pi_h\bvarphi+\gamma_T \H^\text{inc}}
 	\right)^{\frac{1-\alpha}{2}}\bd_{\ba} \, \mathrm d\bx
 	\\ & \le
 	\epsilon \norm{\left(\abs{\bvarphi\tauh+\gamma_T \H^\text{inc}}+\abs{\Pi_h\bvarphi+\gamma_T \H^\text{inc}}
 		\right)^{\frac{\alpha -1}{2}} \left(\bvarphi\tauh-\Pi_h\bvarphi\right)}^2_{\bL^2(\Gamma)}
 	\\ & 
 	\quad\quad+C\int_{\Gamma} \left(\abs{\bvarphi\tauh+\gamma_T \H^\text{inc}}^{1-\alpha}+\abs{\Pi_h\bvarphi+\gamma_T \H^\text{inc}}^{1-\alpha}
 	\right)\abs{{\bd}_{\ba}}^2 \, \,\mathrm d\bx.
 	\end{align*}
 	The constant $\epsilon$ is chosen small enough for the first summand to be absorbed in the left-hand side. 
 	
 	All quantities in the factor multiplied with the defect in the integrand are bounded by the stated regularity assumptions on $\bvarphi$ and $\gamma_T\Hinc$, with the exception of $\bvarphi\tauh$. Consequently, these understood terms are dropped and the rest of the proof focuses on this critical factor containing $\bvarphi\tauh$. The key to estimate this term are the already established stability bounds of Proposition~\ref{prop:phi-bound-tau}. Summation on both sides and applying the Hölder inequality, with the parameters $p=\frac{1+\alpha}{1-\alpha}$ and $q= \frac{1+\alpha}{2\alpha}$, repeatedly in space and time yields for the remaining terms
 	\begin{align*}
 	 \sum_{n=0}^\infty \rho^n\int_{\Gamma} \abs{\bvarphi\tauhn}^{1-\alpha}\abs{\bd_{\ba}^n}^2 \, \,\mathrm d\bx
 	& \le
  \sum_{n=0}^\infty \rho^n \norm{\bvarphi\tauhn}^{1-\alpha}_{\bL^{1+\alpha}(\Gamma)}
 	\norm{\bd^n_{\ba} }^2_{\bL^{\frac{1+\alpha}{\alpha}}(\Gamma)}
 	\\ & \le 
 	\left( \sum_{n=0}^\infty \rho^n \norm{\bvarphi\tauhn}^{1+\alpha}_{\bL^{1+\alpha}(\Gamma)}\right)^{\frac{1-\alpha}{1+\alpha}}
 	\left( \sum_{n=0}^\infty \rho^n
 	\norm{\bd^n_{\ba} }^{\frac{1+\alpha}{\alpha}}_{\bL^{\frac{1+\alpha}{\alpha}}(\Gamma)}\right)^{\frac{2\alpha}{1+\alpha}}.
 	\end{align*}
 	The final remaining factor depending on the numerical solution is bound by the already established bounds from Proposition~\ref{prop:phi-bound-tau}. 
 	
 	 An error rate in terms of the mesh width $h$ is now readily obtained by applying the Hölder continuity of $\ba$. For  $\bvarphi\tauh$ of the stated regularity we obtain 
 \begin{align*}
\norm{\bd_{\ba} }^{\frac{1+\alpha}{\alpha}}_{\bL^{\frac{1+\alpha}{\alpha}}(\Gamma)}
&= \norm{\ba\left(\Pi_h\bvarphi+\gamma_T \H^\text{inc} \right)-
	\ba\left(\bvarphi+\gamma_T \H^\text{inc} \right) }^{\frac{1+\alpha}{\alpha}}_{\bL^{\frac{1+\alpha}{\alpha}}(\Gamma)}
\\ & \le
\norm{\Pi_h\bvarphi -\bvarphi }^{1+\alpha}_{\bL^{1+\alpha}(\Gamma)} 
\le
C\norm{\Pi_h\bvarphi -\bvarphi }^{1+\alpha}_{\bL^{2}(\Gamma)} 
\le C h^{(1+\alpha)(k+1)}.
 \end{align*}
 Inserting this estimate above yields the stated result.
 \end{proof}

\brev
\begin{remark}\label{rem:other-RK-kvgz}
The above theorem generalizes to convolution quadrature time discretizations based on the more general class of stiffly accurate A-stable Runge--Kutta methods. This includes the Lobatto IIIC methods, which have stage order $m-1$ and full classical order $2m-2$ (c.f. \cite{HairerWannerII}).
\end{remark}
\erev

  \subsection{Pointwise error bounds}
  In the context of retarded boundary integral equations error estimates are often shown for points $\bx$ away from the boundary.  
  To derive such bounds with reasonable error rates, the following approach is taken. Firstly, alternate error bounds for the boundary densities are derived. 

  Those results bound the error of approximations of the densities $\bvarphi\tauh$ and a slightly modified electric trace $\widetilde \bpsi\tauh$  with regards of the norms $\bL^{1+\alpha}(\Gamma)$ and $\bL^{\frac{1+\alpha}{\alpha}}(\Gamma)$, which are the natural norms for the present setting. Employing the time-harmonic bounds of the potential operators described in Lemma~\ref{lem:pointwise-potential-bound} then yields pointwise error bounds away from the boundary.

\begin{theorem}\label{them: pointwise}
	Consider the setting of Theorem~\ref{them:semi-discr-time-experimental} under the assumptions stated therein and let $\Gamma$ further be smooth. Furthermore, consider the alternative approximations $(\widetilde\bvarphi\tauh,\widetilde\bpsi\tauh)$ of the boundary densities, derived from the fully discrete solution $\bvarphi\tauh$ and  defined through
	$$\widetilde \bvarphi\tauh = \bvarphi\tauh , \quad \quad
	\widetilde \bpsi\tauh = \ba(\bvarphi\tauh+\gamma_T \Hinc)\times\bnu +\gamma_T \Einc.$$ 
	Fully discrete electromagnetic fields $\E\tauh$ and $\H\tauh$ are then defined for these boundary densities through the discrete representation formulas.
	These  numerical solutions then fulfill, for any $\bx\in\Omega$ away from the boundary, the error bound
  	\begin{align*}
  	&\left(\tau \sum_{n=0}^N \abs{
  		\E\tauhn(\bx)-\E(\bx,\vtn)}^2
  	+
  	\abs{\H\tauhn(\bx)-\H(\bx,\vtn)
  	}^2 \right)^{1/2}
  	\le C \tau^{-1}\left(\tau^{m}+h^{\alpha(k+1)} \right)^{\frac{2\alpha}{1+\alpha}},
  	\end{align*}
  	where the constant $C$ depends on higher Sobolev norms of the exact solution, the boundary $\Gamma$, the point $\bx$, on $\alpha$ and polynomially on the final time $T$. \brev  In the formulations of the error bounds, we use again the shorthand $ \vtn = (t_n+c_i\tau)_{i=1}^m$ to denote the evaluation at the stages of the underlying Runge--Kutta method. \erev To give this simplified version of the error bound, the mild mesh width restriction \brev $h^{4\alpha(k+1)}\le C \tau$ \erev was assumed.
\end{theorem}
\begin{proof}
(i) \emph{Properties of $\ba^{-1}$} \\
The nonlinearity $\ba^{-1}(\bx) = \abs{\bx}^{\frac{1-\alpha}{\alpha}}\bx$ fulfills, by \cite[Lemma 2.3.16]{B15}, the following positivity property 
\begin{align*}
(\ba^{-1}(\bx)-\ba^{-1}(\by) )\cdot (\bx-\by) \ge c\abs{\bx-\by}^{\frac{1+\alpha}{\alpha}}.
\end{align*}
In particular, this identity implies a positivity condition of $\ba$, which reads
\begin{align*}
(\bx-\by)\cdot(\ba(\bx)-\ba(\by) ) 
&= (\ba^{-1}(\ba(\bx))-\ba^{-1}(\ba(\by)) )\cdot(\ba(\bx)-\ba(\by) )
\\ &\ge c\abs{\ba(\bx)-\ba(\by)}^{\frac{1+\alpha}{\alpha}}.
\end{align*}
Furthermore, applying the fundamental theorem of calculus yields, by using the closed form \eqref{eq:Da-explicit-form} of the Jacobian of the nonlinearity $ \ba^{-1}$ , directly the following bound from above  
\begin{align*}
	\abs{\ba^{-1}(\bx)-\ba^{-1}(\by)}
\le \int_0^1\abs{ \boldsymb D \ba^{-1}_{\bx+\theta(\by-\bx)} }\abs{\bx-\by}\mathrm d\theta
\le C \abs{\bx-\by}\left( \abs{\bx}^{\frac{1-\alpha}{\alpha}}+\abs{\bx-\by}^{\frac{1-\alpha}{\alpha}} \right).
\end{align*}
(ii) \emph{Convergence of densities in $\bL^p$-spaces} \\
The identities for the inverse of the nonlinearity $\ba$ imply convergence results for the boundary densities in their respective $\bL^p$-setting, by modifying the proof of Theorem~\ref{them:semi-discr-time-experimental}.

 Applying the positivity to the nonlinear term on the left-hand side of the error equation \eqref{eq:error-equation} yields
\begin{align*}
 	&\left(\bvarphi^{\tau,h}_n-\Pi_h\bvarphi(t_n), \ba\left(\bvarphi^{\tau,h}_n+\gamma_T \H^\text{inc}_n \right)-\ba\left(\Pi_h\bvarphi(t_n)+\gamma_T \H^\text{inc}_n \right) \right)_\Gamma
 	\\& \quad\ge
	\norm{ \ba\left(\bvarphi^{\tau,h}_n+\gamma_T \H^\text{inc}_n \right)-\ba\left(\Pi_h\bvarphi(t_n)+\gamma_T \H^\text{inc}_n \right)}^{\frac{1+\alpha}{\alpha}}_{\bL^{\frac{1+\alpha}{\alpha}}(\Gamma)}.
\end{align*}
The error analysis of Theorem~\ref{them:semi-discr-time-experimental} then yields, under the stated conditions there, a constant $C$ independent of $h$ and $\tau$, such that
\begin{align*}
\tau\sum_{n=0}^N\norm{ \ba\left(\bvarphi^{\tau,h}_n+\gamma_T \H^\text{inc}_n \right)-\ba\left(\Pi_h\bvarphi(t_n)+\gamma_T \H^\text{inc}_n \right)}^{\frac{1+\alpha}{\alpha}}_{\bL^{\frac{1+\alpha}{\alpha}}(\Gamma)}
\le C\left(\tau^m+h^{\alpha(k+1)}\right)^{2}.
\end{align*}
Furthermore, the defect due to the projection $\Pi_h$ is bounded due to the Hölder  continuity of the nonlinearity $\ba$ by
\begin{align*}
\norm{\ba\left(\Pi_h\bvarphi+\gamma_T \H^\text{inc} \right)
	-\ba\left(\bvarphi+\gamma_T \H^\text{inc} \right)}_{\bL^{\frac{1+\alpha}{\alpha}}(\Gamma)}
\le 
\norm{\Pi_h\bvarphi 
	-\bvarphi }^\alpha_{\bL^{1+\alpha}(\Gamma)}
\le Ch^{\alpha(k+1)}.
\end{align*}
Successively applying the Hölder inequality and inserting the boundary condition \eqref{eq:gibc} into $\bpsi$ yields finally
\begin{align}\label{est: psi-tilde-error}
\begin{split}
&\left(\tau\sum_{n=0}^N\norm{\widetilde\bpsi\tauhn-\bpsi(\vtn)}^{2}_{\bL^{\frac{1+\alpha}{\alpha}}(\Gamma)}\right)^{\frac{1}{2}}
 \le
  T^{\frac{1-\alpha}{1+\alpha}}
\left(\tau\sum_{n=0}^N\norm{\widetilde\bpsi\tauhn-\bpsi(\vtn)}^{\frac{1+\alpha}{\alpha}}_{\bL^{\frac{1+\alpha}{\alpha}}(\Gamma)}\right)^{\frac{\alpha}{1+\alpha}}
\\ &
=
 T^{\frac{1-\alpha}{1+\alpha}}
 \left(\tau\sum_{n=0}^N\norm{\ba(\bvarphi\tauhn+\gamma_T \Hinc)-\ba(\bvarphi(\vtn)+\gamma_T \Hinc)}^{\frac{1+\alpha}{\alpha}}_{\bL^{\frac{1+\alpha}{\alpha}}(\Gamma)}\right)^{\frac{\alpha}{1+\alpha}}
\\ & \le
 C\left(\tau^m+h^{\alpha(k+1)}\right)^{\frac{2\alpha}{1+\alpha}}.
\end{split}
\end{align}
We turn towards the estimation of the error of $\bvarphi\tauh $, in terms of the  $\bL^{1+\alpha}(\Gamma)$ norm, which is bounded from above via
\begin{align*}
\abs{\bvarphi\tauh-\bvarphi}
&=
\abs{\ba^{-1}(\ba(\bvarphi\tauh+\gamma_T\Hinc))-\ba^{-1}(\ba(\bvarphi+\gamma_T\Hinc))}
\\  &\le C
\abs{\ba(\bvarphi\tauh+\gamma_T\Hinc)-\ba(\bvarphi+\gamma_T\Hinc)}
\abs{\ba(\bvarphi+\gamma_T\Hinc)}^{\frac{1-\alpha}{\alpha}}
\\ &+C\abs{\ba(\bvarphi\tauh+\gamma_T\Hinc)-\ba(\bvarphi+\gamma_T\Hinc)}^{\frac{1}{\alpha}}.
\end{align*}
Note that the nonlinear defect arising on the right-hand side is simply the pointwise error of $\widetilde\bpsi\tauh$. Taking both sides to the power of $1+\alpha$ and integrating over the boundary $\Gamma$ yields 
\begin{align*}
\norm{\bvarphi\tauh-\bvarphi}_{\bL^{1+\alpha}(\Gamma)}^{1+\alpha}
& \le
C\norm{\abs{\widetilde\bpsi\tauh-\bpsi}
	\abs{\bvarphi+\gamma_T\Hinc}^{1-\alpha} }^{1+\alpha}_{\bL^{1+\alpha}(\Gamma)}
\\ &+
C\norm{\widetilde\bpsi\tauh-\bpsi}^{\frac{1+\alpha}{\alpha}}_{\bL^{\frac{1+\alpha}{\alpha}}(\Gamma)}.
\end{align*}
The first summand is effectively rewritten by plugging the boundary condition into $\bpsi$, which gives the estimate
\begin{align*}
\norm{\abs{\widetilde\bpsi\tauh-\bpsi}
	\abs{\bvarphi+\gamma_T\Hinc}^{1-\alpha} }^{1+\alpha}_{\bL^{1+\alpha}(\Gamma)}
&= \int_\Gamma 
\abs{\widetilde\bpsi\tauh-\bpsi}^{1+\alpha}
\abs{\bvarphi+\gamma_T\Hinc}^{1-\alpha^2}  \mathrm d \bx
\\ & \le 
\norm{
\abs{\widetilde\bpsi\tauh-\bpsi}^{1+\alpha}}_{\bL^{\frac{1}{\alpha}}(\Gamma)}
\norm{\abs{\bvarphi+\gamma_T\Hinc}^{1-\alpha^2} }_{\bL^{\frac{1}{1-\alpha}}(\Gamma)}
\\ & =
\norm{\widetilde\bpsi\tauh-\bpsi}^{1+\alpha}_{\bL^{\frac{1+\alpha}{\alpha}}(\Gamma)}
\norm{\bvarphi+\gamma_T\Hinc }^{1-\alpha^2}_{\bL^{1+\alpha}(\Gamma)}.
\end{align*}
The factor depending only on the exact solution is  independent of $h$ and $\tau$ and bounded due to the regularity assumptions on $\bvarphi$. These estimates imply bounds on the error of the numerical approximation of the boundary density $\bvarphi$ via 
\begin{align*}
\left(\tau\sum_{n=0}^N\norm{\bvarphi\tauhn-\bvarphi(\vtn)}^{2}_{\bL^{1+\alpha}(\Gamma)}\right)^{\frac{1}{2}} 
&\le C\left(\tau\sum_{n=0}^N\norm{\widetilde\bpsi\tauhn-\bpsi(\vtn)}^{2}_{\bL^{\frac{1+\alpha}{\alpha}}(\Gamma)}\right)^{\frac{1}{2}} 
\\ &+
C\left(\tau\sum_{n=0}^N\norm{\widetilde\bpsi\tauhn-\bpsi(\vtn)}^{\frac{2}{\alpha}}_{\bL^{\frac{1+\alpha}{\alpha}}(\Gamma)}\right)^{\frac{1}{2}}.
\end{align*}
The first summand on the right-hand side has already been bounded by \eqref{est: psi-tilde-error}. As a consequence of Minkowski's inequality for $p = \frac{2}{1+\alpha}$, the following estimate holds
\begin{align*}
\tau^{\frac{2}{1+\alpha}}\sum_{n=0}^N\norm{\widetilde\bpsi\tauhn-\bpsi(\vtn)}^{\frac{2}{\alpha}}_{\bL^{\frac{1+\alpha}{\alpha}}(\Gamma)}
&\le
\left(\tau\sum_{n=0}^N\norm{\widetilde\bpsi\tauhn-\bpsi(\vtn)}^{\frac{1+\alpha}{\alpha}}_{\bL^{\frac{1+\alpha}{\alpha}}(\Gamma)}\right)^{\frac{2}{1+\alpha}}
\\ & \le
C\left(\tau^m+h^{\alpha(k+1)}\right)^{\frac{4}{\alpha+1}}.
\end{align*}
Rearranging and taking the square root on both sides yields the estimate
\begin{align*}
\left(\tau\sum_{n=0}^N\norm{\widetilde\bpsi\tauhn-\bpsi(\vtn)}^{\frac{2}{\alpha}}_{\bL^{\frac{1+\alpha}{\alpha}}(\Gamma)} \right)^{\frac{1}{2}}
&\le
C\tau^{\frac{\alpha-1}{2(1+\alpha)}}\left(\tau^m+h^{\alpha(k+1)}\right)^{\frac{2}{1+\alpha}}
\\ & \le 
C\left(\tau^m+h^{\alpha(k+1)}\right)^{\frac{2\alpha}{1+\alpha}},
\end{align*}
where the final estimate holds for the \brev mild mesh width restriction $h^{4\alpha(k+1)}\le C\tau $\erev, since then
\begin{align*}
\tau^{\frac{\alpha-1}{2(1+\alpha)}}\left(\tau^m+h^{\alpha(k+1)}\right)^{\frac{2}{1+\alpha}-\frac{2\alpha}{1+\alpha}}
& \le
2 \left(\tau^{-\frac{1}{2}}\left(\tau^{2m}+h^{2\alpha(k+1)}\right) \right)^{\frac{1-\alpha}{1+\alpha}} \le C.
\end{align*}
Overall, we obtain the complete error bound 
\begin{align}\label{complete-error-bound-Lp}
\begin{split}
	\left(\tau\sum_{n=0}^N\norm{\bvarphi\tauhn-\bvarphi(\vtn)}^{2}_{\bL^{\frac{1+\alpha}{\alpha}}(\Gamma)}\right)^{\frac{1}{2}}
	&+ 	\left(\tau\sum_{n=0}^N\norm{\widetilde\bpsi\tauhn-\bpsi(\vtn)}^{2}_{\bL^{\frac{1+\alpha}{\alpha}}(\Gamma)}\right)^{\frac{1}{2}}
	\\ &\le C\left(\tau^m+h^{\alpha(k+1)}\right)^{\frac{2\alpha}{1+\alpha}}.
\end{split}
\end{align}
(iii) \emph{Pointwise error bound}
Finally, the bound of Lemma~\ref{lem:pointwise-potential-bound}, formulated for the combined block potential operator defined in \eqref{def:W} reads
\begin{align*}
	\norm{\mathbfcal W_{\bx}(s)}_{\mathbb C^3 \times \mathbb C^3 \leftarrow \bL^{1+\alpha}(\Gamma)\times \bL^{\frac{1+\alpha}{\alpha}}(\Gamma)} \le C \abs{s}.
\end{align*}
The stated result is now given by \cite[Lemma 5.2]{BLM11} in combination with the error bound \eqref{complete-error-bound-Lp}.
\end{proof}

\section{\brev Numerical experiments \erev}

The proposed scheme \eqref{eq:full-dis} has been realized in Python by making use of the boundary element library Bempp \cite{bempp} to discretize the boundary integral operators with Raviart-Thomas elements. All codes used to generate the figures of this section are distributed through a Github repository \footnote{\href{https://github.com/joergnick/cqExperiments}{https://github.com/joergnick/cqExperiments}, last accessed on 06/01/2023.}

Consider two unit cubes, shifted from the origin, such that they are separated by a gap of length $l=0.5$. An electric planar wave illuminates the scatterers, and has the closed form
\begin{align}
\label{eq:incident-wave}
\Einc(t,x)=e^{- c(t-x_3-t_0)^2}\be_1,
\end{align}
with the orientation of $\boldsymbol{e_1} =(1,0,0)^T, \,t_0=-2$ and  $c = 100$. This \brev incident \erev wave is scattered from \brev the \erev cubes, where the nonlinear boundary condition with $\alpha = 0.5$ is employed at the boundary. We observe the scattered wave until the final time $T=3$, which captures most of the interaction of the incident field with the cubes.

In this setting, we employ the full discretization of the boundary integral equation and evaluate the time-discrete representation formula to obtain approximations of the scattered wave away from the boundary. 

In order to quantify the error of this approximation, the electric field has been computed at a single point away from the boundary (for our computations we used the origin $\boldsymb P=(0,0,0)$). 
The error is then estimated by computing a reference solution, for which $N=256$ time steps using the $3$-stage Radau IIA based Runge--Kutta convolution quadrature method were used in combination with a $0$-th order Raviart--Thomas boundary element discretization with $6228$ degrees of freedom, which corresponds to the mesh width $h=2^{-7/2}$. Mutually fixing the spatial resolution $h$, or respectively the time step size $\tau$, then yields the convergence plots in Figure~\ref{fig:space_conv} and Figure~\ref{fig:time_conv}.  

\brev \subsection*{Comparison to Theorem~\ref{them: pointwise}} The pessimistic error bounds of Theorem~\ref{them: pointwise} are, in this context, of the order $\mathcal O(\tau^{-1}(\tau^{4/3}+h^{1/3}))$. The numerical evidence indicates that the error behaviour is more accurately described by the convergence rates of the boundary densities of Theorem~\ref{them:semi-discr-time-experimental}, which are of the order $\mathcal O(\tau^2+h^{1/2})$. The order reductions for point evaluations, which are present in the error bounds described in Theorem~\ref{them: pointwise}, are not observed for the present example and might be an artifact of the analysis.
We note that the regularity assumptions of the solutions formulated in Theorem~\ref{them:semi-discr-time-experimental} \& \ref{them: pointwise} are unlikely to hold due to the low regularity of the scatterers. 

Moreover, the space convergence rate seems to be higher than the order $\mathcal O (h^{1/2})$, which would be the expected convergence order of the boundary densities. \erev The sharp increase in accuracy for the final data points might be explained by the comparable parameters used to compute the reference solution. The observed time convergence is more accurately captured by the theory, although the predicted order reductions for point evaluations (from Theorem~\ref{them: pointwise}) are not observed either. Asymptotically, the errors seem to approach the \brev stage \erev order $\mathcal O (\tau^2)$, though the empirical convergence orders are below this expected order of convergence. 

Overall, good convergence properties are observed despite the low regularity of the scatterer. Moreover, this convergence behavior is observed despite many underlying approximations during the implementation, such as the quadrature for the boundary integral operators, hierarchical matrix compression, iterative solution of the linear systems, Newton's method to solve the nonlinear system at each time step and the trapezoidal rule underlying the convolution quadrature method.
\brev 
\subsection*{Scattering from a sphere}
The influence of the parameters $\sigma$ (in the shifted formulation \eqref{bie-weak-shifted} and $\alpha$ is investigated in Figures~\ref{fig:alphas}--\ref{fig:shifts}. 
The incoming wave
\begin{align}
\label{eq:incident-wave-2}
\Einc(t,x)=\sin (c_1(x_3-t +t_0) )e^{-c_2 (t-x_3-t_0)^2}\be_1,
\end{align}
with the constants $c_1 = 20,c_2=2$ and $t_0 = -4$,
illuminates the unit sphere. We observe the scattering until the final time $T=6$. The following two experiments have been conducted with this setup.

\subsection*{On the parameter $\sigma$} In Figure~\ref{fig:shifts}, we fix a coarse grid with $72$ degrees of freedom ($h=0$) and compute several approximations with varying time step sizes and shifts in \eqref{bie-weak-shifted}. We realize the error norm of Theorem~\ref{them:semi-discr-time-experimental} by the coercivity property of the time-harmonic Calderón operator with a fixed frequency (in our experiments $s=1$) and then compute the errors of the traces by using a reference solution with $N=256$ and $\sigma = 1/3$. We then plot the time convergence plot for $\sigma\in\{0,1/3,2/3,1\}$ respectively. The empirical time convergence rate is remarkably close to the predicted convergence rate of Theorem~\ref{them:semi-discr-time-experimental} and independent of $\sigma$, which indicates that the numerical shift discussed in Remark~\ref{rem:shift} is not necessary for practical computations. Moreover, the results indicate that the time convergence rate of Theorem~\ref{them:semi-discr-time-experimental} is optimal.

\subsection*{On the parameter $\alpha$} We fix the $3$-stage Radau IIA time discretization with $N=16$ time steps, which already yields a reasonable approximation, as indicated by Figure~\ref{fig:shifts}. The numerical solution of the fully discrete scheme is computed with sequence of grids with decreasing mesh widths $h_j=2^{-j/2}$ for $j=0,...,4$ and $\alpha \in \{1/3,2/3,1\}$ respectively. The numerical solution is compared with a reference solution, that has been computed by setting $h_{\text{ref}}=2^{-7/2}$. The error is averaged over several points around the sphere, namely the set 
$$ \mathcal P = \left\{\pm c_p\be_1\,,\,\pm c_p\be_2,\pm c_p\be_3 \right\},
$$
with $c_p = 6/5$. As a consequence, the points in $\mathcal P$ are relatively close to the boundary of the scatterer, but not so close that the quadrature of the potential operator becomes a computational challenge. The empirical convergence rates of Figure~\ref{fig:alphas} are faster than the theoretical results from Theorem~\ref{them: pointwise} and moreover independent of $\alpha$. This faster convergence order might be explained by additional structure of the power-law type nonlinearity, that is not used in the proofs of Theorem~\ref{them:semi-discr-time-experimental}--\ref{them: pointwise} (in particular, $\ba$ is locally Lipschitz continuous when restricted to a domain away from the origin by at least a fixed distance ).

\erev

The paper concludes with Figure~\ref{fig:frames}, a visualization of \brev the solution of the nonlinear scattering problem used for the full discretization experiments of Figures~\ref{fig:space_conv}--\ref{fig:time_conv}, \erev which shows the $z=0.5$ plane at several time points.

\begin{figure}[htbp]
	\centering
	\includegraphics[trim = 0mm 0mm 0mm 0mm, clip,width=1.0\textwidth,height=0.6875\textwidth]{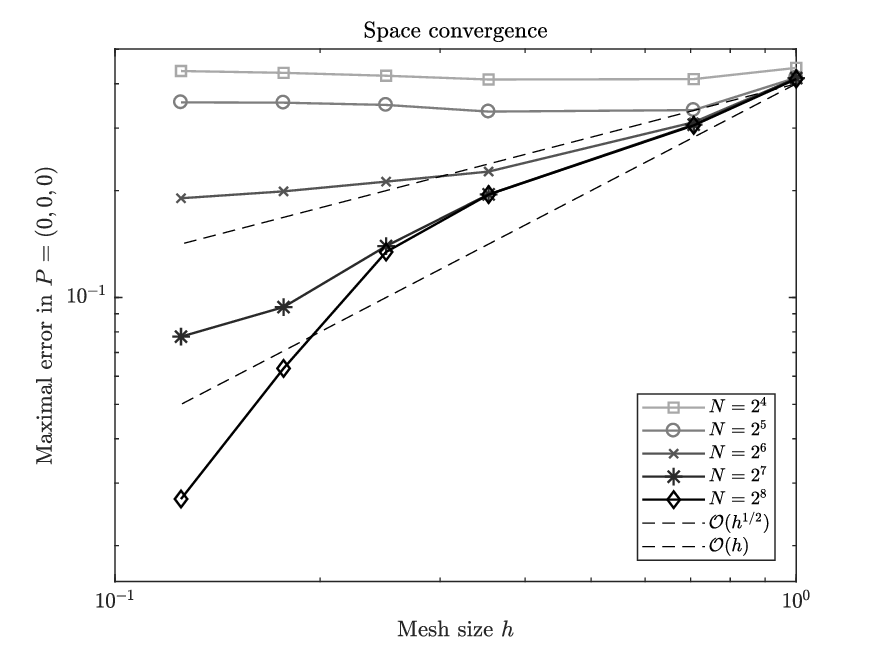}
	\caption{Space convergence plot of the fully discrete system, for $0$th order Raviart--Thomas boundary elements and the $2$-stage Radau IIA based Runge--Kutta convolution quadrature method. 
	The grids generated by the mesh generator embedded in the boundary element library Bempp coincide for $h=2^{-1/2}$ and $h=2^{-1}$, which causes the larger gap between the $4$-th and the $5$-th data point.}
	\label{fig:space_conv}
\end{figure}
\begin{figure}[htbp]
	\centering
		\includegraphics[trim = 0mm 0mm 0mm 0mm, clip,width=1.0\textwidth,height=0.6875\textwidth]{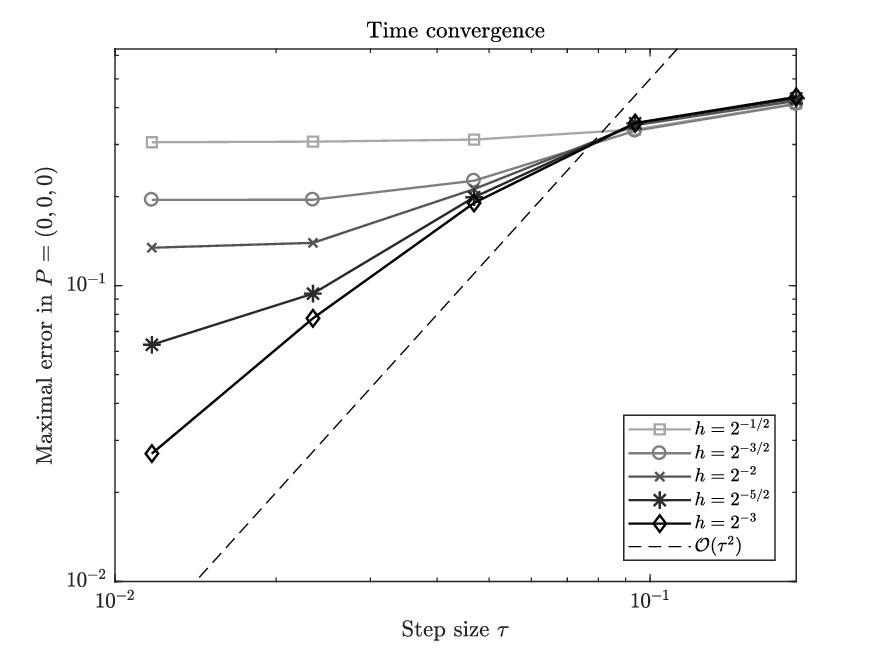}
	\caption{Time convergence plot of the fully discrete system, for $0$-th order Raviart--Thomas boundary elements and the $2$-stage Radau IIA based Runge--Kutta convolution quadrature method. }
	\label{fig:time_conv}
\end{figure}
\begin{figure}[htbp]
	\centering
	\includegraphics[trim = 0mm 0mm 0mm 0mm, clip,scale=0.8]{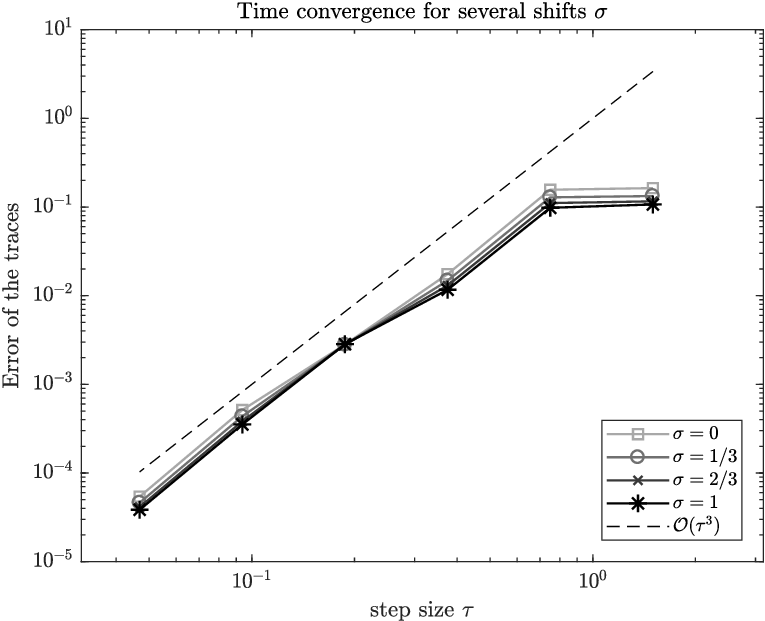}
	\caption{\brev Temporal semi-discretization error plot for scattering from a sphere with $m=3$. The shifted formulation \eqref{bie-weak-shifted} has been used with several shifts $\sigma$. \erev } \vspace{0.1cm}
	\label{fig:shifts}
\end{figure}
\begin{figure}[htbp]
	\centering
	\includegraphics[trim = 0mm 0mm 0mm 0mm, clip,scale=0.8]{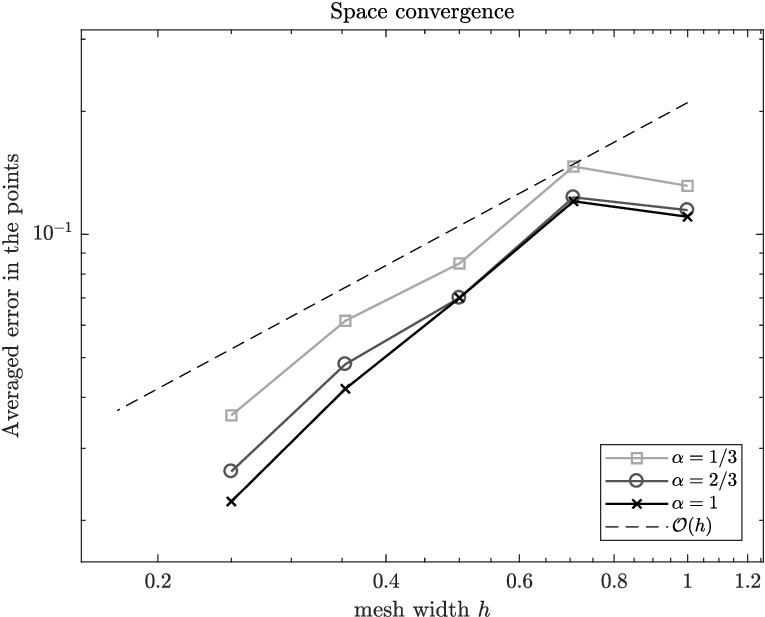}
	\caption{\brev Space convergence rates of the system with a time discretization, for $0$-th order Raviart--Thomas boundary elements with several values of \erev $\alpha$. 
}
	\label{fig:alphas}
\end{figure}

\newpage
\begin{figure}[htp]
	\centering
	\includegraphics[trim = 1mm 12mm 0mm 8mm, clip,scale=1.2]{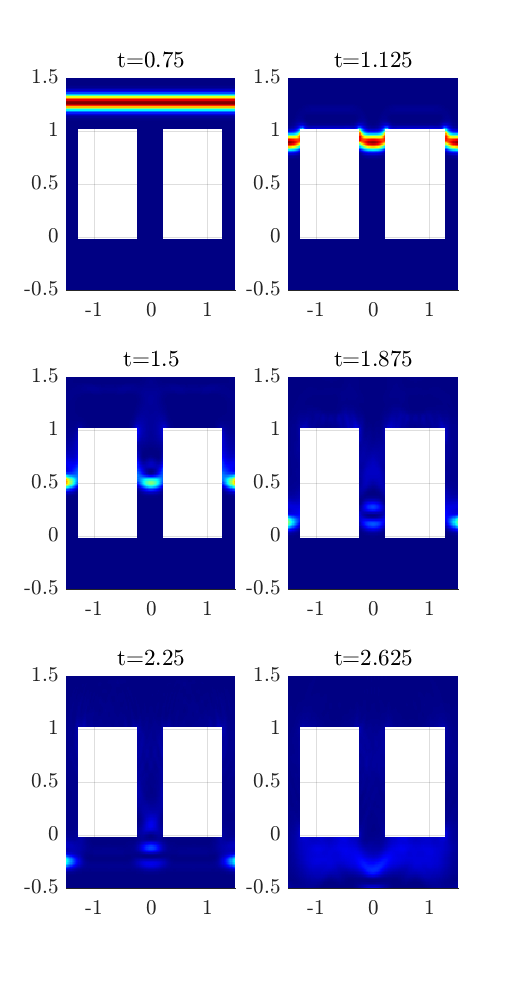}
	\\
	\caption{3D-scattering arising from two unit cubes, with a gap between them at several time points.  The approximation to the scattered wave was computed with $N=256$ time steps until the final time $T=3$ using the $3$-stage Radau IIA method, in combination with $0$th-order Raviart-Thomas boundary elements corresponding to the mesh width $h=2^{-5/2}$, which results in a boundary element space with $1620$ degrees of freedom. }
	\label{fig:frames}
\end{figure}

\section*{Acknowledgment}
The author is very grateful for the continuous support from Christian Lubich during the writing of the manuscript.
This work has been funded by the Deutsche Forschungsgemeinschaft (DFG, German Research Foundation) - Project-ID 58734477 - SFB 1173.  

\bibliographystyle{abbrv}
\bibliography{Lit}

\end{document}